\documentclass[12pt]{amsart}
\usepackage{amssymb,amsfonts,amsmath}
\usepackage{hyperref}

\usepackage[margin=0.7in]{geometry}

\usepackage{mathrsfs}
\usepackage[all]{xy}

\newcommand{\cC}{\mathcal{C}}
\newcommand{\cA}{\mathcal{A}}
\newcommand{\cF}{\mathcal{F}}
\newcommand{\cM}{\mathcal{M}}
\newcommand{\cB}{\mathcal{B}}
\newcommand{\cD}{\mathcal{D}}
\newcommand{\cO}{\mathcal{O}}
\newcommand{\cI}{\mathcal{I}}

\newcommand{\cH}{\mathcal{H}}
\newcommand{\cJ}{\mathcal{J}}
\newcommand{\cL}{\mathcal{L}}
\newcommand{\cP}{\mathcal{P}}

\renewcommand{\ss}{\mathbf{s}}
\newcommand{\cc}{\mathbf{c}}
\newcommand{\hh}{\mathbf{h}}

\newcommand{\sln}{\mathfrak{sl}}

\newcommand{\fc}{\mathfrak{c}}
\newcommand{\fp}{\mathfrak{p}}
\newcommand{\fS}{\mathfrak{S}}
\newcommand{\fh}{\mathfrak{h}}
\newcommand{\fg}{\mathfrak{g}}

\newcommand{\fP}{\mathfrak{P}}

\newcommand{\CC}{\mathbb{C}}
\newcommand{\QQ}{\mathbb{Q}}

\newcommand{\ZZ}{\mathbb{Z}}

\DeclareMathOperator{\fin}{fin}

\DeclareMathOperator{\Ind}{Ind}
\DeclareMathOperator{\Gr}{Gr}
\DeclareMathOperator{\modd}{-mod}
\DeclareMathOperator{\KZ}{KZ}

\DeclareMathOperator{\Vect}{Vect}
\DeclareMathOperator{\GL}{GL}
\DeclareMathOperator{\gln}{\mathfrak{gl}}
\DeclareMathOperator{\Spec}{Spec}
\DeclareMathOperator{\Hom}{Hom}
\DeclareMathOperator{\gr}{gr}
\DeclareMathOperator{\End}{End}
\DeclareMathOperator{\reg}{reg}

\DeclareMathOperator{\Res}{Res}
\DeclareMathOperator{\rank}{rank}
\DeclareMathOperator{\Supp}{Supp}
\DeclareMathOperator{\Irr}{Irr}
\DeclareMathOperator{\HC}{HC}
\DeclareMathOperator{\Ann}{Ann}
\DeclareMathOperator{\Loc}{Loc}

\DeclareMathOperator{\CoInd}{CoInd}
\DeclareMathOperator{\opp}{opp}
\DeclareMathOperator{\cont}{cont}

\DeclareMathOperator{\Tr}{Tr}

\newtheorem{theorem}{Theorem}[section]

\newtheorem{prop}[theorem]{Proposition}
\newtheorem{definition}[theorem]{Definition}
\newtheorem{cor}[theorem]{Corollary}
\newtheorem{lemma}[theorem]{Lemma}
\newtheorem{example}[theorem]{Example}

\author{Huijun Zhao}
\address{Department of Mathematics, Northeastern University. Boston, MA 02115. USA.}
\email{zhao.huij@husky.neu.edu}

\title[Cyclotomic Rational Cherednik Algebras with Aspherical Parameters]{Representations of Cyclotomic Rational Cherednik Algebras with Aspherical Parameters}

\begin{document}
\maketitle

\begin{abstract}
	In this article, we describe all two sided ideals of a cyclotomic rational Cherednik algebra $H_\cc$ and its spherical subalgebra $eH_\cc e$ with a Weil generic aspherical parameter $\cc$, and further describe the simple modules in the category $\cO_\cc^{sph}$. The main tools we use are categorical Kac-Moody actions on catogories $\cO_\cc$ and restriction functors for Harish-Chandra bimodules. 
\end{abstract}

\tableofcontents
\section{Introduction}
Rational Cheredinik algebras $H_\cc(W)$ (also known as rational DAHA)  were introduced by Etingof and Ginzburg in \cite[Section 4]{eg}. They are defined for complex reflection groups $W$ depending on a parameter $\cc$, which is a conjugation invariant function on the set of complex reflections in $W$. In this article we only consider the algebras defined for {\it cyclotomic} complex reflection groups $G(\ell,1,n)$ (see Section \ref{rca} for the explicit definitions).

For a Zariski open subset of all possible parameters $\cc$, $H_\cc$ is Morita equivalent to the {\it spherical subalgebra} $eH_\cc e$, where $e=\frac{1}{|W|}\sum_{w\in W}w\in H_\cc$ denotes the averaging idempotent. Those $\cc$ are called {\it spherical}. In this article, we study the structure of $H_\cc$ with {\it aspherical} (i.e., not spherical) $\cc$, and obtain some results about representations of $eH_\cc e$.

More precisely, we study subcategories of $H_\cc\modd$ and $eH_\cc e\modd$, the category $\cO_\cc(W)$ and $\cO_\cc^{sph}(W)$ (recalled in Section \ref{O} and \ref{sphO}, respectively), analogous to the classical BGG category $\cO$ for semisimple Lie algebras. In fact, there is a functor $H_\cc\modd\to eH_\cc e\modd$ given by $M\mapsto eM=M^W$, restricting to $\cO_\cc(W)\to \cO_\cc^{sph}(W)$, which induces a surjective map on the Grothendieck groups. Moreover, the simples in $\cO_\cc(W)$ are completely known (to be recalled in Section \ref{O}). So we will focus on identifying the kernel of the above functor, i.e., the modules annihilated by $e$.

Losev computed in \cite{Lsupport} the supports for all simple modules in $\cO_\cc(W)$ (further elaborated by Gerber in \cite{ge}, Jacon-Lecouvey in \cite{jl} and Geber-Norton in \cite{gn}), using the Kac-Moody action (constructed by Shan in\cite{s}) and the Heisenberg action (constructed by Shan and Vasserot in\cite{sv}) on the category $\cO_\cc(W)$, together with some combinatorial techniques. The classification of all {\it singular} simples (i.e., those having 0-dimensional supports) follows from the computation.

On the other hand, the spherical subalgebra $eH_\cc(W)e$ of $H_\cc(W)$ has another realization related to quantized quiver varieties.

Let $Q$ be a quiver, $v$ be a dimension vector and $w$ be a framing. The cotangent bundle of the coframed representation space $T^*R=T^*R(Q,v,w)$ has a symplectic structure and a natural $G=\Pi_{i}\GL_{v_i}$-action on it. Therefore the quantum Hamiltonian reduction, with the parameter $\lambda\in \CC^{Q_0}$ (with $\theta$ being a {\it generic} character of $G$, to be elaborated in Section \ref{qqv})
$$ \cA^\theta_\lambda(v,w):=[D(R)/D(R)\{x_R-\langle\lambda,x\rangle|x\in\fg\}|_{T^*R^{\theta-ss}}]^G $$ 
is a sheaf on the Nakajima quiver variety $\cM^\theta_0(v,w)$, whose algebra of global sections is denoted as 
$$ \cA_\lambda(v,w)=\Gamma(\cA^\theta_\lambda(v,w)). $$
 
A computation of $\dim K_0(\cA_\lambda(v,w)\modd_{\fin})$, i.e., the number of all simple finite dimensional modules over $\cA_\lambda(v,w)$, is known when $Q$ is of finite type or of affine type with parameter $\lambda\in\CC^{Q_0}$ being {\it regular}, i.e., the algebra $\cA_\lambda(v,w)$ has finite homological dimension, see \cite{bl}. The derived global section functor $R\Gamma_\lambda^\theta:D^b(\cA_\lambda^\theta(v,w)\modd)\to D^b(\cA_\lambda(v,w)\modd)$, as McGerty and Nevins checked in \cite{mn}, is an equivalence if and only if $\lambda$ is regular. It is expected that $R\Gamma_\lambda^\theta$ is a quotient functor for a singular $\lambda$. 

In the Cherednik case (when $Q$ is cyclic of affine type A  with certain $v$ and $w$ to be specified in Section \ref{cm}), $\cA_\lambda(v,w)\cong eH_\cc e$, where the parameter $\cc$ can be recovered from $\lambda$ (see Theorem \ref{quancm} below), and there is a derived equivalence  $D^b(\cA_\lambda^\theta(v,w)\modd)\simeq D^b(H_\cc\modd)$ (\cite[Section 5]{gl}). It follows from \cite[Theorem 5.5]{e} that $\lambda$ is singular if and only if $\cc$ is aspherical. Furthermore, for cyclotomic rational Cherednik algebras, Dunkl and Griffeth (\cite{dg}) give a description of the aspherical locus of the parameters $\cc$. Namely, the locus of all aspherical $\cc=(c_0,d_j, j\in\ZZ)$ with $d_j=d_{j'}$ if $j=j'$ (mod $\ell$) is a finite union of hyperplanes of the following forms (see Theorem \ref{dg}). 
\begin{itemize}
	\item $c_0=-\frac{k}{m}$ for integers $k$ and $m$ satisfying $1\leq k<m\leq n$ or
	\item there is an integer $j$ with $1\leq j\leq\ell-1$, an integer $m$ with $-(n-1)\leq m\leq (n-1)$ and an integer $k$ such that $k\neq 0$ (mod $\ell$), and
	$$ k=d_j-d_{j-k}+\ell mc_0 \text{, and } 1\leq k\leq j+\lfloor \sqrt{n+\frac{1}{4}m^2}-\frac{1}{2}m-1 \rfloor\ell. $$ 
\end{itemize}

As an application of the restriction functor on Harish-Chandra bimodules over (families of) quantized quiver varieties constructed in \cite[Section 3.3]{bl}, we get a description of all two-sided ideals in $eH_\cc e$ for a {\it Weil generic} aspherical $\cc$ (meaning $\cc$ lies outside countably many algebraic subvariety of the aspherical locus). We will see that the two-sided ideals form a chain (Corollary \ref{cor}) and every two-sided ideal is the annihilator of some simple module in category $\cO$. Using this we will give a description of the simples killed by $e$. 

Our main result is the following.
\begin{theorem}
	For the cyclotomic RCA $H_\cc(G(\ell,1,n))$, with the parameter $$\cc=(c_0,d_0,d_1,\cdots,d_{\ell-1})$$ being Weil generic on the aspherical hyperplane $d_i-d_j-\ell mc_0=k$ $(m,k\in\ZZ)$, the averaging idempotent $e\in H_\cc(G(\ell,1,n))$ annihilates all simples in the category $\cO_\cc$ whose supports are of dimensions not exceeding $n-(p+1)(|m|+p+1)$, where $p=\max\{0,\frac{i-j-k}{\ell}\}$ with $\frac{i-j-k}{\ell}\in\ZZ$ by Theorem \ref{dg}. The kernel of the map $K_0(\cO_\cc)\twoheadrightarrow K_0(\cO_\cc^{sph})$ is generated by all classes of such simples. 
\end{theorem}

The paper is organized as follows. In Chapter 2, we introduce the definitions and properties of the cyclotomic rational Cherednik algebras and of the quantized quiver varieties. Also we recall the description of the aspherical parameters from \cite{dg}, and describe the aspherical locus using different systems of parameters. In Chapter 3, we recall the categorical Kac-Moody actions and the related crystal structures. These two chapters contain no new results. Chapter 4 is about the techniques to compute the supports of the simples in $\cO_\cc$. We apply these techniques to identify all singular simples, which will lead to the description of two-sided ideals in $H_\cc$. Chapter 5 is about the description of two-sided ideals in $eH_\cc e$. We first describe the ideals in the algebra of (twisted) differential operators on the Grassmannian varieties (Theorem \ref{regulargrass}, \ref{singulargrass}) and then prove that they are in bijection with the ideals in $eH_\cc e$ (Theorem \ref{idealbijection}). The main tool we use is the restriction functor for Harish-Chandra bimodules. Finally in Chapter 6, we compare the ideals in $H_\cc$ and in $eH_\cc e$, and obtain the main result.

 \subsection*{Acknowledgement}  
 This paper would have never appeared without the help of Ivan Losev. The author is very grateful to him for countless discussions, remarks and suggestions. The author also would like to thank Jose Simental for stimulating and helpful discussions.

\section{Cherednik algebras and quantized quiver varieties}\label{notation}
In this section, we introduce the notation and the preliminaries. In particular, we recall the description of the aspherical locus of Cherednik algebras (Theorem \ref{dg}), and the correspondence between parameters of the Cherednik algebras, and of the quantized quiver varieties (Theorem \ref{classcalcm} and \ref{quancm}).

 \subsection{Cyclotomic rational Cherednik algebras} \label{rca}
   Let $\fh$ be a complex vector space equipped with an action of a finite {\it complex reflection group} $W$. By a complex reflection group, we mean a group generated by the subset $S\subseteq W$ consisting of {\it complex reflections}, i.e., the elements $s\in W$ such that $\rank(s-1)=1$.
     
   A {\it rational Cherednik algebra} (RCA) depends on the parameters $t\in \CC$ and $\cc:S\to\CC$, a conjugation invariant function. We denote by $\fc$ the space of all possible parameters $\cc$. Clearly $\fc$ is of dimension $|S/W|$, the number of $W$-conjugacy classes in $S$.
    
   For $s\in S$, let $\alpha_s\in\fh^*$ and $\alpha_s^\vee\in\fh$ be the eigenvectors for $s$ with eigenvalues different from 1 (analogs to roots and coroots for Weyl groups). We partially normalize them by requesting that $\langle \alpha_s,\alpha_s^\vee\rangle=2$.
    
   Recall that for an algebra $A$ with a finite group $\Gamma$ acting on it, we can form the {\it smash product} algebra $A\#\Gamma$. It is the same as $A\otimes\CC\Gamma$ as a vector space while the product is given by $(a_1\otimes\gamma_1)\cdot(a_2\otimes\gamma_2)=a_1\gamma_1(a_2)\otimes\gamma_1\gamma_2$, where $\gamma_1(a_2)$ denotes the image of $a_2$ under the action of $\gamma_1$. The product is given so that an $A\#\Gamma$-module is the same as a $\Gamma$-equivariant $A$-module.

   Now a RCA $H_{t,\cc}(W,\fh)$ is defined to be the quotient of $T(\fh\oplus \fh^*)\# W$ by the relations
   $$ [x,x']=[y,y']=0; [y,x]=t\langle y,x\rangle-\sum_{s\in S}\cc(s)\langle\alpha_s^\vee,x\rangle \langle\alpha_s,y\rangle s, $$
   where $x,x'\in\fh^*$ and $y,y'\in\fh$. Note that for any $k\in\CC^\times$, $H_{t,\cc}(W,\fh) \cong H_{kt,k\cc}(W,\fh)$. So we can assume $t=0$ or $t=1$. From now on, we will write $H_\cc$ shortly for $H_{1,\cc}(W,\fh)$.
   
   We can take $\fh$ to be a complex vector space of dimension $n$, and $W$ to be the {\it cyclotomic} complex reflection groups  $G(\ell,1,n):=\fS_n\ltimes\ZZ_\ell^n$, where $\fS_n$ denotes the symmetric group of $n$ letters and $\ZZ_\ell\subseteq\CC^\times$ denotes the group of $\ell$-th roots of unity. $\fS_n$ acts on $\ZZ_\ell^n$ via permutation and the semi-direct product is defined accordingly. $W$ acts on $\fh$ in the following way: $\fS_n$ acts via permuting the coordinates and each factor $\ZZ_\ell$ acts on its own copy of $\CC$ by multiplication by roots of unity. The RCA $H_\cc(G(\ell,1,n), \CC^n)$ is called {\it cyclotomic rational Cherednik algebra}.
    
   The algebra $H_\cc$ is filtered with $\deg \fh^*=\deg W=0$, $\deg \fh=1$. We have a natural epimorphism $S(\fh\oplus\fh^*)\#W\to\gr(H_\cc)$. The following fundamental result is due to Etingof and Ginzburg (\cite{eg}).
   \begin{theorem}[PBW Property of RCA]\label{pbw}
   	The epimorphism $S(\fh\oplus\fh^*)\#W\to\gr(H_\cc)$ is an isomorphism.
   \end{theorem}
   The following corollary of Theorem \ref{pbw} was observed in \cite{ggor}.
   \begin{cor}[Triangular Decomposition]\label{triang}
   	The multiplication map $S(\fh)\otimes \CC W\otimes S(\fh^*)\to H_\cc$ is an isomorphism of vector spaces.
   \end{cor}

 \subsection{Aspherical parameters}\label{para} 
   Let $e=\frac{1}{|W|}\sum_{w\in W}w$ be the averaging idempotent in $\CC W\subseteq H_\cc$. The subspace $eH_\cc e\subseteq H_\cc$ is closed under multiplication and $e$ is a unit. We call $eH_\cc e$ the {\it spherical subalgebra} of $H_\cc$. Note that $\gr(eH_\cc e)=e(S(\fh\oplus\fh^*)\#W)e\cong S(\fh\oplus\fh^*)^W$.
    
   For $t=0$, we have the following result (\cite{eg} Theorem 3.1).
   \begin{theorem}[Satake isomorphism]\label{sataka}
   	For any $\cc\in\fc$, denote the center of $H_{0,\cc}$ by $Z_{0,\cc}$. The map $Z_{0,\cc}\to eH_{0,\cc}e$, $z\mapsto z\cdot e$ is a Poisson algebra isomorphism.
   \end{theorem}
   For $t=1$, however, $H_\cc$ (and also $eH_\cc e$) has trivial center.
    
   We have a functor $H_\cc\modd\to eH_\cc e\modd$ given by $M\mapsto eM(=M^W)$. This functor is an equivalence if and only if $H_\cc=H_\cc eH_\cc$. In this case, we say that the parameter $\cc$ is {\it spherical}, and that $\cc$ is {\it aspherical} otherwise.
    
   We introduce several different systems of parameters and the correspondences between them. Then we describe the aspherical locus in different parameterizations.
    
   For $s\in S$, let $\lambda_s$ denote the eigenvalue of $s$ in $\fh^*$ different from 1, and $H=\fh^s$ the reflection hyperplane. Denote by $W_H$ the point-wise stabilizer of $H$ and by $\ell_H$ the order of $W_H$. We can find $h_{H,j}\in\CC$ for $j=0,\cdots,\ell_H-1$, with $h_{H,j}=h_{H',j}$ for any $j$ if $H'\in WH$, such that
   $$ \cc(s)=\sum_{j=1}^{\ell_H-1}\frac{1-\lambda_s^j}{2}(h_{\fh^s,j}-h_{\fh^s,j-1}). $$
   Clearly these $h_{H,0},\cdots,h_{H,\ell_H-1}$ are defined up to a common summand.
    
   We can recover $h$-parameters from the $\cc$-parameters by the formula
   $$ h_{H,j}=\frac{1}{\ell_H}\sum_{s\in W_H\setminus\{1\}}\frac{2c(s)}{\lambda_s-1}\lambda_s^{-j}. $$
   Note that $\sum_{j=0}^{\ell_H-1}h_{H,j}=0$ in this case.
    
   For the cyclotomic group $W=G(\ell,1,n)$, the set $S$ of complex reflections can be described as follows. For $\eta\in\ZZ_\ell$, let $\eta_{(i)}$ denote the $\eta$ in the $i$-th copy of $\ZZ_\ell\subseteq G(\ell,1,n)$. All such elements are reflections. Other reflections are of the form $(ij)\eta_{(i)}\eta_{(j)}^{-1}$, where $(ij)\in\fS_n$, with $i<j$, denotes the transposition permuting $i$ and $j$, and $\eta\in \ZZ_\ell$. Therefore, we have two conjugacy classes of reflection hyperplanes, whose representatives are given by $x_1=0$ and $x_1=x_2$. We set $\kappa=-\cc(s)$ where $s$ is a reflection in $\fS_n$, and $h_j=h_{H,j}$ where $H$ is the hyperplane $x_1=0$. Note that here we have $\ell_H=\ell$. Therefore we have the $h$-parameters $(\kappa,h_0,\cdots,h_{\ell-1})$ obtained from the $\cc$-parameter $\cc:S/W\to \CC$. We assume that $\kappa\neq 0$.
    
   To make the combinatorial notation in Section \ref{crystal} below easier, we introduce another set of parameters $\ss=\{s_0,\cdots,s_{\ell-1}\}$ for $W=G(\ell,1,n)$, such that $h_j=\kappa s_j-j/\ell$. So we also have the $\ss$-parameter $(\kappa,\ss)$.
   
   Finally, Dunkl and Griffeth introduced one more set of parameters (\cite{dg}) $(c_0,d_j:j\in\ZZ)$ where $c_0=-\kappa$,  $d_j=d_{j'}$ if $j=j'$ (mod $\ell$) and when $0\leq j\leq \ell-1$, $d_j=j-\kappa\ell s_j=-\ell h_j$. Furthermore, in terms of these parameter, Dunkl and Griffeth give the explicit description of aspherical parameters (\cite{dg} Theorem 1.1).
   \begin{theorem}\label{dg}
   	The parameter $(c_0,d_j)$ is aspherical if and only if
   	\begin{itemize}
   		\item [(a)]$c_0=-\frac{k}{m}$ for integers $k$ and $m$ satisfying $1\leq k<m\leq n$ or
   		\item [(b)]there is an integer $j$ with $1\leq j\leq\ell-1$, an integer $m$ with $-(n-1)\leq m\leq (n-1)$ and an integer $k$ such that $k\neq 0$ (mod $\ell$), and
   		$$ k=d_j-d_{j-k}+\ell mc_0 \text{, and } 1\leq k\leq j+\lfloor \sqrt{n+\frac{1}{4}m^2}-\frac{1}{2}m-1 \rfloor\ell. $$ 
   	\end{itemize}
   \end{theorem}
   In the $\ss$-parameters, the above conditions can be rephrased as the following (\cite[4.1.2]{gl}). The parameter $(\kappa,s_0,\cdots,s_{\ell-1})$ is aspherical if and only if
   \begin{itemize}
       \item [(a')]$\kappa=\frac{k}{m}$ for integers $k$ and $m$ satisfying $1\leq k<m\leq n$ or
       \item [(b')]there is an integer $j$ with $1\leq j\leq\ell-1$, an integer $m$ with $-(n-1)\leq m\leq (n-1)$ and an integer $k$ such that $k\neq 0$ (mod $\ell$), and
      	$$ k-\hat{k}=\kappa\ell(s_{j-k}-s_j-m) \text{, and } 1\leq k\leq j+\lfloor \sqrt{n+\frac{1}{4}m^2}-\frac{1}{2}m-1 \rfloor\ell, $$
      	where $j+1-\ell\leq\hat{k}\leq j$ and $\hat{k}=k$ (mod $\ell$).
   \end{itemize}
   
  We remark that when $\ell=1$, we have $W=\fS_n$. In this case, $H_\cc$ is called the RCA of type A, and the aspherical parameters $\cc$ are exactly the rational numbers in the interval $(-1,0)$ with denominators not exceeding $n$. Both the structure theory of $H_\cc$ and the representation theory of $eH_\cc e$ have been well-studied (see, for instance, \cite{eg}, \cite{ggor}, \cite{beg}, \cite{gs}). In this article, we assume that $\ell>1$. 
  
  We also assume that the aspherical parameter $\cc$ is {\it Weil generic} (meaning that $\cc$ lies outside countably many algebraic subvarieties) in one of the hyperplanes listed in Theorem \ref{dg} (b) (or (b')).
     
  \subsection{Quantized Nakajima quiver varieties}\label{qqv}
   A spherical cyclotomic RCA can be also realized as a quantized Nakajima quiver variety. Under the realization, aspherical parameters go to singular parameters. 
  
   Let $Q=(Q_0,Q_1,t,h)$ be a quiver (which is an oriented graph and we allow loops and multiple edges), where $Q_0$ is a finite set of vertices, $Q_1$ a finite set of arrows and $t,h:Q_1\to Q_0$ are maps assigning to each arrow its tail and head. Pick vectors $v,w\in\ZZ^{Q_0}_{\geq 0}$ and vector spaces $V_i,W_i$ for $i\in Q_0$ such that $\dim V_i=v_i$ and $\dim W_i=w_i$. We can construct the (co)framed representation space 
   $$ R=R(Q,v,w):=\bigoplus_{a\in Q_1}\Hom(V_{t(a)},V_{h(a)})\oplus\bigoplus_{i\in Q_0}\Hom(V_i,W_i). $$
   We will consider the cotangent bundle $T^*R=R\oplus R^*$, which can be identified with
   $$ \bigoplus_{a\in Q_1}(\Hom(V_{t(a)},V_{h(a)})\oplus\Hom(V_{h(a)},V_{t(a)}))\oplus\bigoplus_{i\in Q_0}(\Hom(V_i,W_i)\oplus\Hom(W_i,V_i)). $$
   The space $T^*R$ carries a natural symplectic structure, denoted by $\omega$. On $R$ we have a natural action of the group $G:=\Pi_{i\in Q_0}\GL(v_i)$. This action extends to an action on $T^*R$ by linear symplectomorphisms, admitting a {\it moment map}, that is, a $G$-equivariant morphism $\mu:T^*R\to\fg^*$, with $\fg$ the Lie algebra of $G$, such that $\{\mu^*(x),\cdot\}=x_{T^*R}$ for any $x\in\fg$, where $\mu^*:\fg\to \CC[T^*R]$ is the dual map to $\mu$, $\{\cdot,\cdot\}$ is the Poisson bracket on $\CC[T^*R]$ induced by $\omega$ and $x_{T^*R}$ is a vector field on $T^*R$ induced by the $G$-action. Also there is an action of a one-dimensional torus $\CC^\times$ on $T^*R$ given by $t.r=t^{-1}r$. We specify the moment map uniquely by requiring that it is quadratic: $\mu(t.r)=t^{-2}\mu(r)$. In this case $\mu^*(x)=x_{R}$, where we view $x_R\in \Vect_R$ as a function on $T^*R$.
    
   Now we proceed to the definition of Nakajima quiver varieties. Pick a character $\theta$ of $G$ and also an element $\lambda\in(\fg/[\fg,\fg])^*$. To $\theta$ we associate an open set $(T^*R)^{\theta-ss}$ of $\theta$-semistable points in $T^*R$ (which may be empty). Recall that a point $r\in T^*R$ is called $\theta$-{\it semistable} if there is a $(G,n\theta)$-semi-invariant polynomial $f\in \CC[T^*R]$ (i.e., $f(g^{-1}r)=\theta(g)^nf(r)$ for any $g\in G$ and any $r\in T^*R$) with $n>0$ such that $f(r)\neq 0$. A Nakajima quiver variety is defined as the GIT quotient $$\cM_\lambda^\theta(v,w)=(\mu^{-1}(\lambda)\cap(T^*R)^{\theta-ss})//G.$$
   This space is smooth provided $(\lambda,\theta)$ is {\it generic}, i.e. when the $G$-action on $\mu^{-1}(\lambda)^{\theta-ss}:=\mu^{-1}(\lambda)\cap(T^*R)^{\theta-ss}$ is free. A description of generic values of $(\lambda,\theta)$ is due to Nakajima (\cite{n1}). Namely, $(\lambda,\theta)$ is generic if there is no $v'\in\ZZ_{\geq 0}^{Q_0}$ such that 
   \begin{itemize}
   	\item $v'\leq v$ (component-wisely),
   	\item $\sum_{i\in Q_0} v_i'\alpha_i$ is a root for the Kac-Moody algebra $\fg(Q)$ associated to the underlying graph of $Q$, where $\alpha_i$, $i\in Q_0$, are the simple roots for $\fg(Q)$ corresponding to $i$.
   	\item and $v'\cdot \theta=v'\cdot\lambda=0$.
   \end{itemize}
   All varieties $\cM_\lambda^\theta(v,w)$ carry natural Poisson structures as Hamiltonian reductions. For a generic pair $(\lambda,\theta)$, $\cM_\lambda^\theta(v,w)$ is symplectic.
   
   The variety $\cM_\lambda^0(v,w)$ is affine and there is a projective morphism $\rho:\cM_\lambda^\theta(v,w)\to \cM_\lambda^0(v,w)$. Denote $\cM(v,w):=\Spec(\CC[\cM_0^\theta(v,w)])$ for a generic $\theta$ (i.e., $(0,\theta)$ is generic). Note that the variety $\cM(v,w)$ is independent of the choice of $\theta$. When the moment map $\mu$ is flat, $\cM_0^0(v,w)=\cM(v,w)$ and $\rho:\cM_0^\theta(v,w)\to \cM(v,w)$ is a conical symplectic resolution, with a $\CC^\times$-action coming from the dilation action on $T^*R$ (\cite[Section 2]{bpw}). According to \cite[Theorem 1.1]{cb}, $\mu$ is flat if and only if
   $$ p(v)+w\cdot v-(w\cdot v^0+\sum_{i=0}^k p(v^i))\geq 0 $$
   for any decomposition $v=v^0+v^1+\cdots+v^k$, or equivalently, for the decompositions where $v^1,\cdots,v^k$ are roots. Here $p(v):=1-\frac{1}{2}(v,v)$ and $(\cdot,\cdot)$ denotes the {\it symmetrized Tits form} defined by 
   $$(v^1,v^2):=2\sum_{k\in Q_0}v^1_kv^2_k-\sum_{a\in Q_1}(v_{t(a)}^1v^2_{h(a)}+v^1_{h(a)}v^2_{t(a)}).$$
   
   Next, we construct quantizations of $\cM_0^\theta(v,w)$ that will be certain sheaves of filtered algebras on $\cM_0^\theta(v,w)$. Namely, consider the algebra $D(R)$ of differential operators on $R$. We can localize this algebra to a microlocal (the sections are only defined on $\CC^\times$-stable open subsets) sheaf on $T^*R$ denoted by $D_R$. We have a quantum comoment map $\Phi:\fg\to D(R)$ quantizing the classical comoment map $\fg\to \CC[T^*R]$, still $\Phi(x)=x_R$. Now fix $\lambda\in \fP:=(\fg^*)^G\cong\CC^{Q_0}$. We get the quantum Hamiltonian sheaf on $\cM_0^\theta(v,w)$, 
   $$ \cA_\lambda^\theta(v,w):=\pi_*[D_R/D_R\{\Phi(x)-\lambda(x)|x\in\fg\}|_{(T^*R)^{\theta-ss}}]^G, $$
   where $\pi:\mu^{-1}(0)^{\theta-ss}\to \cM_0^\theta(v,w)$ is the quotient map. This is a sheaf of filtered algebra with $\gr\cA_\lambda^\theta(v,w)=\cO_{\cM_0^\theta(v,w)}$ (for generic $\theta$), and hence it has no higher cohomology. The global section $\cA_\lambda(v,w):=\Gamma(\cA_\lambda^\theta(v,w))$ is an algebra with $\gr\cA_\lambda(v,w)=\CC[\cM_0^\theta(v,w)]$. One can prove that $\cA_\lambda(v,w)$ is independent of the choice of $\theta$ (\cite[Corollary 3.8]{bpw}). When $\mu$ is flat, $\cA_\lambda(v,w)$ coincides with $\cA_\lambda^0(v,w)$ (\cite[2.2.2]{bl}).
   
 \subsection{Generalized Calogero-Moser spaces}\label{cm}
   Recall that for the RCA $H_{0,\cc}(W,\fh)$, we have the Satake isomorphism (Theorem \ref{sataka}) $Z_{0,\cc}\cong eH_{0,\cc}e$. The space $C_{\cc}:=\Spec(eH_{0,\cc}e)$ is called the {\it generalized Calogero-Moser space}, of which the type A case was introduced by Kazhdan-Kostant-Sternberg (\cite{kks}) and Wilson (\cite{w}) in connection with the Calogero-Moser integrable systems. \par 
   There is also a realization of the generalized Calogero-Moser spaces $C_\cc(W)$ with $\cc=(c_0,d_0,\cdots,d_{\ell-1})$ for $W=G(\ell,1,n)$ as Nakajima quiver varieties. Take the quiver $Q$ of type $\tilde{A}_{\ell-1}$ and the dimension vectors $v=n\delta$, where $\delta=(1,\cdots,1)$ corresponds to the indecomposable imaginary root of $Q$, and $w=\epsilon_0$, the simple root at the extending vertex labeled by $0$. Given $\lambda\in \CC^{Q_0}$ recovered from $\cc$ (to be specified below), one can get the Nakajima quiver variety
   $$ \cM_\lambda^0(v,w)=\mu^{-1}(\lambda)//G, $$
   the quantized quiver variety $\cA^\theta_\lambda(v,w)$, and its global section $\cA_\lambda(v,w)$.
  
   The result in \cite[3.10]{g} can be rephrased as follows.
   \begin{theorem} \label{classcalcm}
   	Let $\cc$, $v$, $w$ be as above, and $$\lambda^c=\frac{1}{\ell}(\ell c_0-d_0+d_{\ell-1}, d_0-d_1,d_1-d_2,\cdots,d_{\ell-2}-d_{\ell-1}),$$ where the superscript $c$ stands for classical. Then there is an isomorphism between the Poisson varieties
   	$$ C_\cc=\Spec(eH_{0,\cc} e)\cong \cM_{\lambda^c}^0(v,w). $$
   \end{theorem}
   On the quantized level, we have the following result for the spherical subalgebra, rephrased from \cite[1.4, 3.6]{g2} and \cite[1.4]{eggo}. 
   \begin{theorem}\label{quancm}
   	Let $\cc$, $v$, $w$ be as above, and
   	$$ \lambda^q=\frac{1}{\ell}(1-\ell(c_0+1)+d_0-d_{\ell-1}, 1-d_0+d_1,1-d_1+d_2,\cdots, 1-d_{\ell-2}+d_{\ell-1}), $$
   	where the superscript $q$ stands for quantum.
   	Then there is an algebra isomorphism 
   	$$ \cA_{\lambda^q}(v,w)\cong eH_\cc e. $$ 	
   \end{theorem}

\section{Categorical actions on Cherednik category $\cO$} \label{categorification}
In this section, we continue working with the RCA setting. The triangular decomposition (Corollary \ref{triang}) allows us to define the category $\cO_\cc$ for $H_\cc$. The category $\cO_\cc$ enjoys many nice properties. In particular, several functors give rise to a categorical Kac-Moody action on $\cO_\cc$. This action is essential in Section \ref{support}.

 \subsection{Category $\cO_\cc(W)$}\label{O}
 The category $\cO_\cc(W)$ is defined as the subcategory of $H_\cc(W)\modd$ consisting of all modules with locally nilpotent action of $\fh\subseteq H_c$ that are finitely generated over $H_\cc$, or equivalently (under the condition that $\fh$ acts locally nilpotently), over $S(\fh^*)\subseteq H_\cc$.
  
 One can define Verma modules over $H_\cc$, 
 which are parameterized by the irreducible representations of $W$: given $\tau\in\Irr(W)$, set $\Delta_\cc(\tau):=H_\cc\otimes_{S(\fh)\#W}\tau$, where $\fh$ acts on $\tau$ by 0. Thanks to the triangular decomposition, the Verma module $\Delta_\cc(\tau)$ is naturally isomorphic to $S(\fh^*)\otimes\tau$ as an $S(\fh^*)\# W$ module (where $W$ acts diagonally and $S(\fh^*)$ acts by multiplications on the left). We identify $K_0(\cO_\cc(W))$ to $K_0(W\modd)$ by sending the class $[\Delta_\cc(\tau)]$ to $[\tau]$. Furthermore, there is a partial order on the set of all simples in $\cO_\cc(W)$, to be specified in Section \ref{hw}.
  
 Let us now describe the irreducible representations of $W=G(\ell,1,n)$, and hence the simples in $\cO_\cc(W)$. The set $\Irr(W)$ is in a natural bijection with the set $\cP_\ell(n)$ of $\ell$-partitions $\nu=(\nu^{(0)},\cdots, \nu^{(\ell-1)})$ of $n$, satisfying $\sum_{i=0}^{\ell-1}|\nu^{(i)}|=n$. The irreducible module $V_\nu$ corresponding to $\nu\in \cP_\ell(n)$ can be constructed as follows. The product $G(\ell,1,\nu):=\Pi_{i=0}^{\ell-1}G(\ell,1,|\nu^{(i)}|)$ naturally embeds into $G(\ell,1,n)$. Let $V_{\nu^{(i)}}$ denote the irreducible $\fS_{|\nu^{(i)}|}$-module labeled by the partition $\nu^{(i)}$. We equip $V_{\nu^{(i)}}$ with the structure of a $G(\ell,1,|\nu^{(i)}|)$-module by making all $\eta_{(j)}$ act by $\eta^i$. Denote the resulting $G(\ell,1,|\nu^{(i)}|)$-module by $V_{\nu^{(i)}}^{(i)}$. Let $V_\nu$ denote the $G(\ell,1,n)$-module induced from the $G(\ell,1,\nu)$-module $V_{\nu^{(0)}}^{(0)}\boxtimes V_{\nu^{(1)}}^{(1)}\boxtimes\cdots\boxtimes V_{\nu^{(\ell-1)}}^{(\ell-1)}$. The modules $V_\nu$ form a complete collection of the irreducible $W$-modules.
  
 There is a so-called {\it Euler element} $\hh\in H_\cc$ satisfying $[\hh,x]=x$, $[\hh,y]=-y$, and $[\hh,w]=0$ for $x\in\fh^*$, $y\in\fh$, $w\in W$. It is constructed as follows. Take a basis $y_1,\cdots,y_n\in\fh$ and the dual basis $x_1,\cdots,x_n\in \fh^*$. For $s\in S$, let $\lambda_s$ be the eigenvalue of $s$ in $\fh^*$ different from 1, as in \ref{para}. Then 
 $$ \hh=\sum_{i=1}^nx_iy_i+\frac{n}{2}-\sum_{s\in S}\frac{2\cc(s)}{1-\lambda_s}s. $$ 
 Using the Euler element one can establish the following standard structural results about the categories $\cO_\cc(W)$. 
 \begin{prop}\label{simple}
 	Every Verma module $\Delta_\cc(\tau)$ has a unique simple quotient, to be denoted by $L_\cc(\tau)$. The map $\tau\mapsto L_\cc(\tau)$ is a bijection between $\Irr(W)$ and $\Irr(\cO_\cc(W))$.
 \end{prop}
 Also we can establish the following result about the structure of $\cO_\cc(W)$(\cite[Section 2]{ggor}).
 \begin{prop}\label{proj}
 	The category $\cO_\cc(W)$ has enough projectives and all objects there have finite length. Each projective object admits a finite filtration whose successive quotients are Verma modules.
 \end{prop}

\subsection{Highest weight structure}\label{hw}
The categories $\cO_\cc(W)$ are highest weight categories, analogous to the classical BGG category $\cO$. This result was established in \cite{ggor}.\par 
Let us start by recalling the general notion of a highest weight category over the field $\CC$. Let $\cC$ be a $\CC$-linear abelian category equivalent to $A\modd$ for some finite dimensional algebra $A$. Equip $\Irr(\cC)$ with a partial order $\leq$. For $L\in\Irr(\cC)$, let $P_L$ denote the projective cover of $L$ and consider the subcategory $\cC_{\leq L}$ consisting of all objects $L'$ in $\cC$ such that $L'\leq L$. The object $\Delta_L$ denotes the maximal quotient of $P_L$ lying in $\cC_{\leq L}$.
\begin{definition}\label{hwdef}
	We say that $\cC$ is a highest weight category (with respect to the order $\leq$) if, for every $L\in\Irr(\cC)$, the kernel of $P_L\twoheadrightarrow \Delta_L$ is filtered by $\Delta_{L'}$ with $L'\in\Irr(\cC)$ and $L'>L$. The objects $\Delta_L$ are called standard.
\end{definition} 
Back to the categories $\cO_\cc(W)$,  Proposition \ref{simple} and Proposition \ref{proj} imply that $\cO_\cc(W)$ is equivalent to the category of modules over the finite dimensional algebra $\End_{\cO_\cc(W)}(P)^{\opp}$, where $P:=\bigoplus_{\tau\in\Irr(W)}P_\cc(\tau)$, and $P_\cc(\tau)$ denotes the projective cover of $L_\cc(\tau)$. Also, we can define a partial order $\leq_\cc$ as follows. First note that the Euler element $\hh$ acts on $\tau\subseteq\Delta_\cc(\tau)$ by a scalar denoted by $\cc_\tau$, which is called the {\it $\cc$-function} of $\tau$. For $W=G(\ell,1,n)$, we can compute the functions $\cc_\nu$ explicitly for $\nu\in\cP_\ell(n)\cong\Irr(W)$. We view the elements of $\cP_\ell(n)$ as $\ell$-tuples of Young diagrams. Let $b$ be a box of $\nu$. It can be characterized by three numbers $x,y,i$, where $x$ is the number of column, $y$ is the number of row, and $i$ is the number of the component $\nu^{(i)}$ containing $b$. Further, we set $\cc_b:=\kappa\ell(x-y)+\ell h_i$ and $\cc_\nu:=\sum_{b\in\nu}\cc_b$. We set $\tau\leq_\cc\xi$ if $\tau=\xi$ or $\cc_\tau-\cc_\xi\in\ZZ_{>0}$. 

The following result is established in \cite[Theorem 2.19]{ggor}.
\begin{prop}\label{hwO}
	For any complex reflection group $W$, the category $\cO_\cc(W)$ is highest weight with respect to the order $\leq_\cc$. The standard $\Delta_{L(\tau)}$ coincides with the Verma module $\Delta_\cc(\tau)$.
\end{prop} 
It turns out that for the cyclotomic group $W=G(\ell,1,n)$, a rougher order than $\leq_\cc$ will also work. Recall the parameters $(\kappa,s_0,\cdots,s_{\ell-1})$. We will write $\cO_{\kappa,\ss}(n)$ for $\cO_\cc(G(\ell,1,n))$. Then $\cc_b=\kappa\ell(x-y-s_i)-i$. We denote $\cont^\ss(b):=x-y+s_i$. Define an equivalence relation on boxes by $b\sim b'$ if $\kappa(\cont^\ss(b)-\cont^\ss(b'))\in \ZZ$. We write $b\preceq_\cc b'$ if $b\sim b'$ and $\cc_b-\cc_{b'}\in \ZZ_{\geq 0}$. Define the order $\preceq_\cc$ on $\cP_\ell$ as $\lambda\preceq_\cc \lambda'$ if one can order boxes $b_1,\cdots,b_n$ of $\lambda$ and $b_1',\cdots,b_n'$ of $\lambda'$ such that $b_i\preceq_\cc b_i'$ for any $i$. Clearly $\lambda\preceq_\cc \lambda'$ implies $\lambda\leq_\cc\lambda'$. The following result is due to Dunkl and Griffeth, \cite[Theorem 1.2]{dg}.
\begin{prop}\label{roughorder}
	One can take $\preceq_\cc$ for a highest weight order for $\cO_\cc(n)$.
\end{prop}
We set $\cO_{\kappa,\ss}:=\bigoplus_{n\geq 0}\cO_{\kappa,\ss}(n)$. We have the basis $|\lambda\rangle:=[\Delta_{\kappa,\ss}(\lambda)]$ in $K_0^\CC(\cO_{\kappa,\ss})$ indexed by $\lambda\in\cP_\ell:=\bigsqcup_{n\geq 0}\cP_\ell(n)$. In other words, $K_0^\CC(\cO_{\kappa,\ss})$ is the level $\ell$ Fock space.\par 
For certain values of $\ss$, the category $\cO_{\kappa,\ss}$ can be decomposed as follows. We define an equivalence relation $\sim_\cc$ on $\{0,\cdots,\ell-1\}$ by setting $i\sim_\cc j$ if the $i$th and $j$th partitions can contain equivalent boxes, i.e., $s_i-s_j\in \kappa^{-1}\ZZ+\ZZ$. For an equivalence class $\alpha$, we write $\ss(\alpha)$ for $(s_i)_{i\in\alpha}$ and $\cP_\alpha$ for the subset of all $\nu\in\cP_\ell$ with $\nu^{(j)}=\emptyset$ for $j\notin\alpha$. Form the category $\bigotimes_\alpha\cO_{\kappa,\ss(\alpha)}$. The simples in this category are labeled by the set $\Pi_\alpha\cP_\alpha$ that is naturally identified with $\cP_\ell$.
\begin{prop}\cite[Lemma 4.2]{Lsupport}\label{blocks}
	There is a highest weight equivalence $\cO_{\kappa,\ss}\xrightarrow{\sim}\bigotimes_\alpha\cO_{\kappa,\ss(\alpha)}$
\end{prop}
Using Proposition \ref{blocks}, we can reduce the study of $\cO_{\kappa,\ss}$ to the case when we have just one equivalence class in $\{0,1,\cdots,\ell-1\}$.
 
\subsection{Functors on $\cO_\cc(W)$}\label{functors}
 We recall several functors from the category $\cO_\cc(W)$, which will be important in the construction of the categorical Kac-Moody action: the localization functor, the KZ functor, and the (parabolic) induction and restriction functors.
 
 \subsubsection{Localization functor}\label{local}
 Let $\fh^{\reg}$ denote the open subset of $\fh$ consisting of all $y\in\fh$ with the point-wise stabilizer $W_y=\{1\}$, equivalently $\fh^{\reg}=\fh\setminus\bigcup_{s\in S}\ker\alpha_s$. Consider an element $\delta\in\CC[\fh]^W$ whose set of zeroes in $\fh$ coincides with $\fh\setminus\fh^{\reg}$. We can take $\delta=(\Pi_{s\in S}\alpha_s)^k$ where $k$ is a suitable integer such that $\delta\in \CC[\fh]^W$. Note that $[\delta,x]=[\delta,w]=0$ for all $x\in \fh^*$, $w\in W$, and $[\delta,y]\in S(\fh^*)\#W$ so $[\delta,[\delta,y]]=0$. It follows that the endomorphism $[\delta,\cdot]$ of $H_\cc$ is locally nilpotent. So the set $\{\delta^k,k\geq 0\}$ satisfies the Ore conditions and we have the localization $H_\cc[\delta^{-1}]$ consisting of right fractions. We have the {\it Dunkl homomorphism} $H_\cc\to D(\fh^{\reg})\# W$ defined on generators $x\in \fh^*$, $w\in W$, $y\in \fh$ as follows
 $$ x\mapsto x, w\mapsto w, y\mapsto y+\sum_{s\in S}\frac{2\cc(s)\langle \alpha_s,y \rangle}{(1-\lambda_s)\alpha_s}(s-1). $$
 This homomorphism factors through $H_\cc[\delta^{-1}]\to D(\fh^{\reg})\# W$ because $\delta$ is invertible in $D(\fh^{\reg})\# W$. The following lemma is easy.
 \begin{lemma}\label{easy}
 	The homomorphism $H_\cc[\delta^{-1}]\to D(\fh^{\reg})\# W$ is an isomorphism.
 \end{lemma}
 Let $M\in \cO_\cc(W)$. Thanks to Lemma \ref{easy}, we can view $M[\delta^{-1}]$ as a module over $D(\fh^{\reg})\#W$. This module is finitely generated over $\CC[\fh^{\reg}]\#W$. So $M[\delta^{-1}]$ is a $W$-equivariant local system over $\fh^{\reg}$. Computing $\Delta_\cc(\tau)[\delta^{-1}]$ explicitly, we see that $M[\delta^{-1}]$ has regular singularities (\cite[Proposition 5.7]{ggor}). Therefore the same is true for any $M$. So we get an exact functor $M\mapsto M[\delta^{-1}]$ from $\cO_\cc(W)$ to the category $\Loc_{rs}^W(\fh^{\reg})$ of $W$-equivariant local systems on $\fh^{\reg}$ with regular singularities.
 \subsubsection{KZ functor}\label{KZ}
 Pick a point $p\in \fh^{\reg}/W$ and let $\pi$ denote the quotient morphism $\fh^{\reg}\to\fh^{\reg}/W$. According to Deligne, there is a category equivalence $\Loc_{rs}^W(\fh^{\reg})\cong\pi_1(\fh^{\reg}/W,p)\modd_{\fin}$, where the latter is the category of finite dimensional modules over the fundamental group of $\fh^{\reg}/W$ with base point $p$, defined as $N\mapsto [\pi_*(N)^W]_p$. The group $\pi_1(\fh^{\reg}/W,p)$ is known as the braid group of $W$ and denoted by $B_W$. In general, $B_W$ is generated by $T_H$, where $H$ runs over the set of reflection hyperplanes for $W$, with certain relations.
 
 We want to determine the essential image of the functor $\cO_\cc(W)\to \Loc_{rs}^W(\fh^{\reg})\xrightarrow{\cong} B_W\modd_{\fin}$. It turns out that this image coincides with that of $\cH_q(W)\modd\hookrightarrow B_W\modd_{\fin}$ (see \cite{e2}), where $\cH_q(W)$, called the Hecke algebra of $W$, is a quotient of $\CC B_W$ by the relations
 $$ \Pi_{i=1}^{|W_H|}(T_H-q_{H,i})=0, $$
 with $H$ running over the set of reflection hyperplanes for $W$, where 
 $$q=\{q_{H,i}:H \text{ is a reflection hyperplane for W, } i=0,1,\cdots,\ell_H-1\}$$
 is a parameter recovered from $\cc$ via
 $$ q_{H,j}=\exp(2\pi\sqrt{-1}(h_{H,j}+\frac{j}{\ell_H})). $$
  
 It was shown in \cite[Theorem 5.13]{ggor} that the functor $\cO_\cc(W)\to B_W\modd_{\fin}$ decomposes as the composition of $\KZ_\cc:\cO_\cc(W)\to \cH_q(W)\modd$ (called the {\it KZ functor}) and the inclusion $\cH_q(W)\modd\hookrightarrow B_W\modd_{\fin}$.
 Let us list some properties of the KZ functor obtained in \cite[Section 5]{ggor}.
 \begin{prop}\label{KZproperty}
 	The following is true.
 	\begin{itemize}
 		\item[(1)] The KZ functor $\cO_\cc(W)\twoheadrightarrow \cH_q(W)\modd$ is a quotient functor. Its kernel is the subcategory $\cO_{\cc,tor}(W)\subseteq \cO_c(W)$ consisting of all modules in $\cO_\cc(W)$ that are torsion over $\CC[\fh]$ (equivalently, whose support is a proper subvariety in $\fh$).
 		\item[(2)] The functor $\KZ_\cc$ is defined by a projective object $P_{\KZ}$ in $\cO_\cc(W)$ that is also injective. The multiplicity of $\Delta_\cc(\tau)$ in $P_{\KZ}$ equals $\dim\tau$.
 		\item[(3)] $\KZ_\cc$ is fully faithful on the projective objects in $\cO_\cc(W)$. 
 		\item[(4)] Suppose that the parameter $q$ satisfies the following condition: for any reflection hyperplane $H$, we have $q_{H,i}\neq q_{H,j}$ for any $i\neq j$. Then $\KZ_\cc$ is fully faithful on all standardly filtered objects.
    \end{itemize}
 \end{prop}
 We remark that Etingof recently proved (\cite{e2}) that $\cH_q(W)$ is of dimension $|W|$. 
 
 \subsubsection{Induction and restriction functors} \label{indres}
 The induction and restriction functors were introduced by Bezrukavnikov and Etingof in \cite{be}. These functors relate categories $\cO_\cc(W)$ and $\cO_\cc(\underline{W})$, where $\underline{W}$ is a parabolic subgroup in $W$. More precisely, we have functors $\Res_W^{\underline{W}}:\cO_\cc(W)\to \cO_\cc(\underline{W})$ and $\Ind_{\underline{W}}^W:\cO_\cc(\underline{W})\to \cO_\cc(W)$, where in $\cO_\cc(\underline{W})$, abusing the notation, $\cc$ means the restriction of $\cc:S\to \CC$ to $\underline{W}\cap S$. Since we also consider restriction and induction functors for other categories, we will sometimes write $^\cO\Res_W^{\underline{W}}$, $^\cO\Ind_{\underline{W}}^W$. The construction of the functors is technical. We are not going to explain the construction, only the properties.
 \begin{prop}\label{beadj}
 	The functors $\Res_W^{\underline{W}}$ and $\Ind_{\underline{W}}^W$ are biadjoint. Hence they are exact.
 \end{prop}
 The claim that $\Res_W^{\underline{W}}$ and $\Ind_{\underline{W}}^W$ are exact was checked in \cite[Section 3.5]{be}. The claim that $\Ind_{\underline{W}}^W$ is right adjoint to $\Res_W^{\underline{W}}$ follows from the construction in {\it loc. cit.}. The other adjointness was established in \cite[Section 2.4]{s} under some restrictions on $W$ and in \cite{Lfunctors} in general.
  
 We have a natural homomorphism $\cH_q(\underline{W})\to \cH_q(W)$. This gives rise to an exact restriction functor $^\cH\Res_W^{\underline{W}}:\cH_q(W)\modd \to \cH_q(\underline{W})\modd$. The following proposition (\cite[Theorem 2.1]{s}) states the relation between restriction functors and KZ functors.
 \begin{prop}\label{resKZ}
 	The KZ functors intertwine the restriction functors: $$\underline{\KZ}_\cc\circ ~^\cO\Res_W^{\underline{W}}\cong ~^\cH\Res_W^{\underline{W}}\circ \KZ_\cc.$$ Here we write $\underline{\KZ}_\cc$ for the KZ functor $\cO_\cc(\underline{W})\to \cH_q(\underline{W})$.
 \end{prop}
 We get the induction $^\cH\Ind_{\underline{W}}^W$, and the coinduction $^\cH\CoInd_{\underline{W}}^W$ functors $\cH_q(\underline{W})\modd\to \cH_q(W)\modd$. As explained in \cite{s}, we have the following corollary.
 \begin{cor}\label{indKZ}
 	We have an isomorphism of functors $^\cH\Ind_{\underline{W}}^W\cong$$^\cH\CoInd_{\underline{W}}^W$. The KZ functors intertwine the induction functors.
 \end{cor}

 \subsection{Categorical Kac-Moody action}\label{km}
 Define the Kac-Moody algebra $\fg_\kappa$ as $\hat{\sln}_e$ if $\kappa$ is rational with denominator $e$ and $\sln_\infty$ if $\kappa$ is irrational. Here $\sln_\infty$ stands for the Kac-Moody algebra of infinite rank associated to the type A Dynkin diagram that is infinite in both directions. When $e=1$, we assume $\fg_\kappa=\{0\}$. Define the Kac-Moody algebra $\fg_{\kappa,\ss}$ as the product of several copies of $\fg_\kappa$, one per equivalence class for $\sim_\cc$ in $\{0,1,\cdots,\ell-1\}$. This algebra has generators $e_z$, $f_z$, where $z$ runs over the subset in $\CC/\kappa^{-1}\ZZ$ of the elements of the form $s_i+m$, where $m$ is an integer. We are going to define an action of $\fg_{\kappa,\ss}$ on $\cF^\ell$, the Fock space of level $\ell$.\par 
 The action is defined as follows. We say that a box $b$ is a $z$-box if $\cont^\ss(b)$ equals to $z$ in $\CC/\kappa^{-1}\ZZ$. We set $f_z|\nu\rangle:=\sum_{\mu}|\mu\rangle$, where the sum is taken over all $\ell$-partitions $\mu$ that are obtained from $\nu$ by adding a $z$-box. Similarly, we set $e_z|\nu\rangle:=\sum_{\lambda}|\lambda\rangle$, where the sum is taken over all $\lambda$ obtained from $\nu$ by removing a $z$-box. We write $\cF_{\kappa,\ss}$ for the space $\cF^{\ell}$ equipped with this $\fg_{\kappa,\ss}$-action. Note that we have a natural isomorphism of $\fg_{\kappa,\ss}$-modules, $\cF_{\kappa,\ss}=\bigotimes_{i=0}^{\ell-1}\cF_{\kappa,s_i}$.\par 
 Now we recall the definition of type A Kac-Moody actions. Let $\cC$ be an abelian $\CC$-linear category, where all objects have finite length. A {\it type A categorical Kac-Moody action} on $\cC$ as defined in \cite[5.3.7, 5.3.8]{r2} consists of the following data:
 \begin{itemize}
 	\item[(1)] exact endo-functors $E$, $F$ of $\cC$ and a number $q\in\CC\setminus\{0,1\}$,
 	\item[(2)] adjointness morphisms $1\to EF$, $FE\to 1$,
 	\item[(3)] endomorphisms $X\in\End(E)$, $T\in\End(E^2)$.  
 \end{itemize}
 These data are supposed to satisfy the axioms to be listed below. We will need the following notation. Let $I$ be a subset in $\CC^\times$. Define a Kac-Moody algebra $\fg_I$ as follows. Define an unoriented graph structure on $I$ by connecting $z$ and $z'$ if $z'z^{-1}=q^{\pm 1}$. Then $\fg_I$ is the Kac-Moody algebra defined from $I$, it is the product of several copies of $\hat{\sln}_e$ if $q$ is a primitive root of unity of order $e$, and is the product of several copies of $\sln_\infty$ otherwise. For example, taking $q=\exp(2\pi\sqrt{-1}\kappa)$ and $I=\{Q_0,Q_1,\cdots,Q_{\ell-1}\}$ where $Q_i:=\exp(2\pi\sqrt{-1}\kappa s_i)$, we get $\fg_I=\fg_{\kappa,s}$. The axioms of a categorical action are as follows.
 \begin{itemize}
 	\item[(i)] $F$ is isomorphic to the left adjoint of $E$.
 	\item[(ii)] For any $d$, the map $X_i\to 1^{i-1}X1^{d-i}$ (where $1^{i-1}X1^{d-i}$ is short for $11\cdots1X11\cdots1\in\End(E^d)$ with $X$ occurring in the $i$-th position and 1 denoting the identity morphism, meaning that $X$ is applied to the $i$-th copy of $E$), $T_i\to 1^{i-1}T1^{d-1-i}$ extends to a homomorphism $\cH_q^{aff}(d)\to \End(E^d)$, where $\cH_q^{aff}(d)$ is the {\it affine Hecke algebra} generated by $X_1,\cdots,X_d$, $X_1^{-1},\cdots,X_d^{-1}$ and $T_1,\cdots,T_{d-1}$ subject to the following relations：
 	\begin{itemize}
 		\item The subalgebra generated by $T_1,\cdots,T_{d-1}$ is isomorphic to $\cH_q(\fS_d)$ defined in \ref{KZ};
 		\item $X_1^\pm,\cdots,X_d^{\pm}$ commute with each other;
 		\item $T_iX_j=X_jT_i$ if $i\neq j, j-1$ and $T_iX_iT_i=qX_{i+1}$.
 	\end{itemize}
 	\item[(iii)] Let $E=\bigoplus_{z\in\CC}E_z$ be the decomposition into eigen-functors according to $X$, and $F=\bigoplus_{z\in\CC}F_z$ be the decomposition coming from (2). The operators $[E_z]$, $[F_z]$ give rise to an integrable representation of $\fg_I$ on $K_0^\CC(\cC)$, where $I:=\{z\in\CC^\times:E_z\neq 0\}$.
 	\item[(iv)] Let $\cC_\nu$ denote the Serre subcategory of $\cC$ spanned by the simples $L$ with $[L]\in K_0^\CC(\cC)_\nu$, where $K_0^\CC(\cC)_\nu$ is the $\nu$ weight space for the $\fg_I$-module $K_0^\CC(\cC)$. Then $\cC=\bigoplus_\nu\cC_\nu$.
 \end{itemize}
 Now let us proceed to constructing a categorical $\fg_{\kappa,\ss}$-action on $\cO_{\kappa,\ss}$ that categorifies the $\fg_{\kappa,\ss}$-action on $\cF_{\kappa,\ss}$, following \cite{s}.\par 
 Set $$E:=\bigoplus_{n=0}^\infty~^\cO\Res_{G(\ell,1,n-1)}^{G(\ell,1,n)}, F:=\bigoplus_{n=0}^\infty ~^\cO\Ind_{G(\ell,1,n+1)}^{G(\ell,1,n)}.$$
 The construction of endomorphisms $X\in \End(E)$, $T\in\End(E^2)$ comes from the following categorification.
 
 There is a categorical $\fg_I$-action on $\cC:=\bigoplus_{n\geq 0}\cH_q^\ss(n)\modd$ (\cite[Section 7.2]{cr}), where $\cH_q^\ss(n)$ is the cyclotomic Hecke algebra, i.e., $\cH_q(W)$ for $W=G(\ell,1,n)$, and $q$ is the eponymous parameter, $I=\{Q_0,\cdots, Q_{\ell-1}\}$ as before. Let $^\cH\Res_{n-1}^n$ denote the restriction functor $\cH_q^\ss(n)\modd\to \cH_q^\ss(n-1)\modd$ (we set $^\cH\Res_{-1}^0=0$) and let $^\cH\CoInd^{n-1}_n$ denote the coinduction functor, the right adjoint of the restriction functor. We set $E:=\bigoplus_{n=0}^\infty$$^\cH\Res_{n-1}^n$ and $F:=\bigoplus_{n=0}^\infty$$^\cH\CoInd_{n+1}^n$. The endomorphism $^\cH X$ on the summand $^\cH\Res_{n-1}^n$ is given by the multiplication by $X_n\in \cH_q^\ss(n)$ and $^\cH T$ on the summand $^\cH\Res_{n-2}^n$ of $E^2=\bigoplus_{n=0}^\infty$ $^\cH\Res_{n-1}^n$ the multiplication by $T_{n-1}$. 
 
 Set the category $\cO_\cc:=\bigoplus_{n\geq 0}\cO_\cc(n)$. Then we have the KZ functor $\KZ_\cc:\cO_\cc\twoheadrightarrow\bigoplus_{n=0}^\infty\cH_q^\ss(n)\modd$, the sum of KZ functors $\cO_\cc(n)\twoheadrightarrow\cH_q^\ss(n)\modd$. By Proposition \ref{resKZ}, $\KZ_\cc$ intertwines both $E$ and $F$. Since $\KZ_\cc$ is fully faithful on the projective objects, it induces an isomorphism $\End(^\cH E)\simeq \End(^\cO E)$ that gives an element $X$ in the right hand side. Similarly we have an isomorphism $\End(^\cH E^2)\simeq \End(^\cO E^2)$ that gives us $T\in \End(^\cO E^2)$.
 
 Note that the eigenvalue decomposition in (iii) is given by $[E_z\Delta_\cc(\nu)]=\sum_\mu[\Delta_\cc(\mu)]$, where the summation is taken over all $\mu$ such that $\mu\subseteq\nu$ and $\nu\setminus\mu$ is a $z$-box. Similarly $[F_z\Delta_c(\nu)]=\sum_ \lambda[\Delta_c(\lambda)]$, where the summation is taken over all $\lambda$ such that $\nu\subseteq\lambda$ and $\lambda\setminus\nu$ is a $z$-box. We further remark here that the categorical $\fg_{\kappa,\ss}$-action on $\cO_\cc$ is compatible with the highest weight structure in the sense of the following lemma (\cite[Lemma 3.3]{Lsurvey}).
 \begin{lemma} \label{hwcompatible}
 	The object $E_z\Delta_\cc(\nu)$ has a filtration by $\Delta_\cc(\mu)$, where $\mu$ runs over all $\ell$-partitions obtained by removing a $z$-box from $\nu$, each $\Delta_\cc(\mu)$ occurs with multiplicity 1. Similarly $F_z\Delta_\cc(\nu)$ has a filtration by $\Delta_\cc(\lambda)$, where $\lambda$ runs over all $\ell$-partitions obtained by adding a $z$-box to $\nu$, each $\Delta_\cc(\lambda)$ occurs with multiplicity 1.
 \end{lemma}
  
 \section{Supports}\label{support}
 In this section we use the categorical Kac-Moody action recalled in Section \ref{km} to study the supports of modules in $\cO_{\kappa,\ss}$.
  
 \subsection{Possible supports}\label{xpq}
 Set $W=G(\ell,1,n)$. Every object $M\in\cO_\cc(W)$ is finitely generated over $S(\fh^*)=\CC[\fh]$. So we can define the support $\Supp(M)$ of $M$ in $\fh$. This will be the support of $M$ viewed as a coherent sheaf on $\fh$. By definition, this is a closed subvariety in $\fh$.
 
 It turns out that $\Supp(M)$ is the union of the strata of the {\it stabilizer stratification} of $\fh$. The strata are labeled by the conjugacy classes of possible stabilizers for the $W$-action on $\fh$ (these stabilizers are exactly the parabolic subgroups of $W$ by definition). Namely, to a parabolic subgroup $\underline{W}\subset W$ we assign the locally closed subvariety $X(\underline{W}):=\{b\in\fh|W_b=\underline{W}\}$. Note that $\overline{X(\underline{W})}:=\bigsqcup_{W'}X(W')$, where the union is taken over the conjugacy classes of all parabolic subgroups $W'$ containing a conjugate of $\underline{W}$.
 
 Let $e$ denote the denominator of $\kappa$, and $p,q$ non-negative integers satisfying $p+eq\leq n$. Set $W_{p,q}:=G(\ell,1,n-p-eq)\times \fS_e^q$. This is the stabilizer of the point $$(x_1,\cdots,x_p,y_1,\cdots,y_1,\cdots,y_q,\cdots,y_q,0,\cdots,0),$$ and hence parabolic, where $x_1,\cdots,x_p,y_1,\cdots,y_q$ are pairwise different complex numbers and each $y_1,\cdots,y_q$ occurs $e$ times. When $e=1$, we assume that $p=0$. When $\kappa$ is irrational, we take $e=+\infty$ and $q$ is automatically 0. Also note that $\overline{X(W_{p,q})}:=W\fh^{W_{p,q}}\subseteq \fh$.
  
 Clearly, for an exact sequence $0\to M'\to M\to M''\to 0$, we have $\Supp(M)=\Supp(M')\cup\Supp(M'')$. This, in principal, reduces the computation of supports to the case of simple modules. The following result is implicit in \cite[Section 3.8]{be} and explicit in \cite[Section 3.10]{sv}.
 \begin{lemma}\label{besupp}
   Let $L\in\Irr(\cO_{\kappa,\ss}(n))$. Then $\Supp(L)=\overline{X(W_{p,q})}$ for some non-negative integers $p$, $q$ satisfying $p+eq\leq n$.
 \end{lemma}  
 For a Weil generic aspherical parameters $\cc$ on the hyperplane $s_i-s_j=m+\frac{t}{\kappa}$, $\kappa$ is irrational. So $q=0$ and $p\leq n$. Furthermore, $\overline{X(W_{p,0})}$ is exactly the subset of $\fh$ consisting of those points with at most $p$ coordinates being nonzero. So $\dim\overline{X(W_{p,0})}=p$. 
  
 In the remainder of this chapter, for $\nu\in\cP_\ell$, we give a technique to compute $p_{\kappa,\ss}(\nu)$ such that $\Supp(L_{\kappa,\ss}(\nu))=\overline{X(W_{p_{\kappa,\ss}(\nu),0})}$, following \cite{Lcrystals}.
 \subsection{Crystals}\label{crystal}
 Let us recall the definition of a crystal corresponding to a categorical $\fg_I$-action. The crystal structure will be defined for each $z\in I$ separately so we will get $I$ copies of an $\sln_2$-crystal. By an $\sln_2$-crystal we mean a set $C$ with maps $\tilde{e},\tilde{f}:C\to C\sqcup\{0\}$ such that 
 \begin{itemize}
 	\item For any $v\in C$ there are integers $m$ and $n$ such that $\tilde{e}^nv=\tilde{f}^mv=0$;
 	\item Moreover, for $u,v\in C$, the equalities $\tilde{e}u=v$ and $\tilde{f}v=u$ are equivalent.
 \end{itemize} 
 Now let $\cC$ be an abelian category equipped with a $\fg_I$-action. We will introduce a $\fg_I$-crystal structure on the set $\Irr(\cC)$, following \cite{cr}. Namely, pick $L\in \Irr(\cC)$ and consider the object $E_zL$. If it is nonzero, it has a simple head (i.e., the maximal semisimple quotient) and a simple socle (i.e., the maximal semisimple subobject) and those two are isomorphic, according to \cite[Proposition 5.20]{cr}. We take that simple object as $\tilde{e}_zL$ if $E_zL\neq 0$. And we set $\tilde{e}_zL=0$ if $E_zL=0$. $\tilde{f}_zL$ is defined similarly. That we get a $\fg_I$-crystal follows from \cite[Proposition 5.20]{cr} combined with \cite[Section 5]{bk}.
  
 Now let us explain how to compute the crystal on $\cP_\ell=\Irr(\cO_\cc)$, following \cite{Lsurvey}. In order to compute $\tilde{e}_z\nu$, $\tilde{f}_z\nu$, we first record the {\it $z$-signature} of $\nu$ that is a sequence of $+$'s and $-$'s. Then we perform a certain reduction procedure getting what we call the {\it reduced $z$-signature}. Based on that signature, we can then compute $\tilde{e}_z\nu$, $\tilde{f}_z\nu$. What we get is a crystal on $\cP_\ell$ very similar to what was discovered by Uglov in \cite{u}.
  
 To construct the $z$-signature of $\nu$, we take the addable and removable $z$-boxes of $\nu$, order them in a decreasing way according to the order $\preceq_\cc$ introduced in Section \ref{hw}, and write a $+$ for each addable box and a $-$ for each removable box. To reduce the signature, we erase consecutive $-+$ leaving empty spaces and continue until there is no $-$'s to the left of a $+$. It is easy to see that the reduced signature does not depend on the order in which we perform the reduction. Finally to obtain $\tilde{e}_z\nu$ (resp. $\tilde{f}_z\nu$), we pick the leftmost $-$ (rightmost $+$) in the reduced $z$-signature of $\nu$ and remove (add) the box in the corresponding position of $\nu$. Here is an example.
 
 \begin{example}
 	Suppose $\ell=2$, $\kappa<0$ irrational, $s_0=0, s_1=1$. We take $z=0$ and $\nu^{(0)}=(2,2), \nu^{(1)}=(2)$. We get boxes $b_1$ (removable) and $b_2$ (addable), with $b_2\preceq_\cc b_1$.
 	
 	\setlength{\unitlength}{1mm}
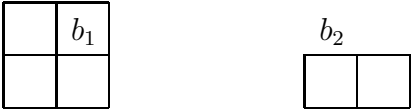
 	\begin{picture}(65,18)
 	\put(2,1){\line(0,1){14}}
 	\put(9,1){\line(0,1){14}}
 	\put(16,1){\line(0,1){14}}
 	\put(2,1){\line(1,0){14}}
 	\put(2,8){\line(1,0){14}}
 	\put(2,15){\line(1,0){14}}
 	\put(42,1){\line(0,1){7}}
 	\put(49,1){\line(0,1){7}}
 	\put(42,1){\line(1,0){14}}
 	\put(42,8){\line(1,0){14}}
 	\put(56,1){\line(0,1){7}}
 	\put(11,10){$b_1$}
 	\put(44,10){$b_2$}
 	\end{picture}
 	
 	The $z$-signature we get is $-+$ and is reduced to empty. So $\tilde{e}_z\nu=0$.
 \end{example}
  
 The $\fg_{\kappa,\ss}$-crystal $\cP_{\kappa,\ss}$ is highest weight, i.e., for every $\nu\in\cP_{\kappa,\ss}$, there is $k\in\ZZ_{\geq 0}$ such that $\tilde{e}_{z_1}\cdots\tilde{e}_{z_k}\nu=0$ for any $z_1,\cdots,z_k\in I$. We define the {\it depth} of $\nu$ as $k-1$ for the minimal such $k$. The following result was obtained in \cite[Section 5.5]{Lcrystals}.
 \begin{prop}\label{p}
 	The number $p_{\kappa,\ss}(\nu)$ coincides with the depth of $\nu$ in $\cP_{\kappa,\ss}$
 \end{prop}
  
 \subsection{Singular irreducible modules}\label{singular}
 In this section, we find all possible finite dimensional irreducible modules $L(\nu)$ over $H_\cc$ for Weil generic aspherical $\cc$. We remark that the same result has been obtained in \cite[Corollary 8.1]{gr}, \cite[Lemma 3.3]{dg} and \cite[Example 3.4, 3.5]{Lsurvey}. 
 
 Singular simples $L(\nu)$ correspond to singular elements $\nu\in\cP_\ell$, i.e., $\tilde{e}_z\nu=0$ for any $z\in \CC/\kappa^{-1}\ZZ$. Note that $\dim\Supp(L(\nu))=0$ by the discussion after Lemma \ref{besupp}. In terms of the $(\kappa,\ss)$-parameter, we consider the aspherical hyperplanes of the form 
 $$ s_i-s_j=m+\frac{t}{\kappa}, $$
 where $m, t\in\ZZ$ and we assume $i<j$. Now we do the computation in four cases.
 
 \textbf{Case 1:} $m\geq 0, t=0$.
 
 We start with the simplest case where the hyperplane is just $s_i-s_j=m\geq 0$. If $|\nu|\neq 0$, there exists at least one removable box in $\nu$. Given that $\tilde{e}_z\nu=0$ for any $z\in \CC/\kappa^{-1}\ZZ$, there have to be at least equally many addable $z$-boxes as removable ones and the addable $z$-boxes have to be smaller than removable ones with respect to $\preceq_\cc$ so that the $-+$ can be reduced.
 
 Since $\kappa\notin \QQ$, $b\sim b'$ for two (possibly addable and not in $\nu$) boxes $b$ and $b'$ of $\nu$ if and only if 
 \begin{itemize}
 	\item one of them lies in the component $\nu^{(i)}$ and the other in $\nu^{(j)}$, to be denoted respectively $b_i$ and $b_j$, and 
 	\item $\cont^\ss(b_i)=\cont^\ss(b_j)$, i.e., $x(b_i)-y(b_i)+m=x(b_j)-y(b_j)$, where $x$ and $y$, as in \ref{hw}, denotes respectively the number of column and row of the box.
 \end{itemize}
 
 Recall that $\cc_{b_i}=\kappa\ell(\cont^\ss(b_i))-i$ and $\cc_{b_j}=\kappa\ell(\cont^\ss(b_j))-j$ with $i<j$. So $\cc_{b_i}>\cc_{b_j}$. In the $s_j$-signature, $b_j$ comes to the left of $b_i$. To produce a $-+$, $b_j$ is forced to be the removable box and $b_i$ addable. When $\kappa\notin \QQ$, $\tilde{e}_z\nu=0$ for any $z\neq s_j \mod \kappa^{-1}\ZZ$ implies that there is no removable boxes other than $b_j$. So any component in $\nu$ other than $\nu^{(j)}$ is empty. In particular, $\nu^{(i)}=\emptyset$ and $x(b_i)-y(b_i)=0$. Then $x(b_j)-y(b_j)=m$. So $\nu^{(j)}$ is a rectangle with $r$ rows ($r\geq 1$) and $r+m$ columns.
  
 In summary, if the parameter $\cc$ is Weil generic on the hyperplane $s_i-s_j=m$ with $m\geq 0$, the singular $\nu\in\cP_\ell$ satisfies $\nu^{(j)}=((r+m)^r)$ with $r\geq 1$, and $\nu^{(j')}=\emptyset$ for any $j'\neq j$.
 
 \textbf{Case 2:} $m<0, t=0$.
 
 For the case $s_i-s_j=m<0$, the argument goes almost the same as in Case 1. The only difference is that now $r+m<r$, so we require $r+m\geq 1$. The conclusion is also similar to that in Case 1, only with $r\geq 1$ replaced by $r+m\geq 1$. After relabeling, one can write $\nu^{(j)}=(r^{r+|m|})$ with $r\geq 1$, and $\nu^{(j')}=\emptyset$ for any $j'\neq j$.
 
 \textbf{Case 3:} $t>0$.
 
 The argument that $b\sim b'$ implies $b=b_i$ and $b'=b_j$ is still true. A minor difference is that now we require that $\cont^\ss(b_i)=\cont^\ss(b_j)\mod \kappa^{-1}\ZZ$. However, this still gives us $x(b_i)-y(b_i)+m=x(b_j)-y(b_j)$. A more significant difference is that when $t\neq 0$,
 $$ \cc_{b_i}-\cc_{b_j}=\kappa\ell(s_i-s_j-m)-(i-j)=t\ell-(i-j). $$
 Given $t>0$ and $-\ell<i-j<0$, we see that $\cc_{b_i}>\cc_{b_j}$. So the result coincides with that of Case 1 if $m\geq 0$, and with Case 2 if $m<0$.
 
 {\textbf Case 4:} $t<0$. 
 
 This case is similar to Case 3 and the only difference is that when $t<0$, $\cc_{b_i}<\cc_{b_j}$ and hence $b_i$ is the only removable box in $\nu$. So in this case $\nu^{(i)}=((r+m)^r)$ if $m\geq 0$ and $(r^{r+|m|})$ if $m<0$ with $r\geq 1$, and $\nu^{(i')}=\emptyset$ for any $i'\neq i$.

 \subsection{Description of supports}\label{supp}
 Based on the computation in \ref{singular}, we can give a more explicit description of the depth of any element in $\cP_\ell$, and hence the supports of simples in $\cO_\cc$, still under the assumption that $\cc$ is Weil generic on the hyperplane $s_i-s_j=m+\frac{t}{\kappa}$. 
 
 \begin{prop}\label{n-rec}
  	For any $\nu\in\cP_{\kappa,\ss}(n)$ with $\cc=(\kappa,\ss)$ Weil generic on the hyperplane $s_i-s_j=m+\frac{t}{\kappa}$, the depth 
 	$$ p(\nu)=n-r(r+|m|) $$
 	for some $r\in\ZZ$ and $0\leq r\leq \lfloor \sqrt{n+\frac{1}{4}m^2}-\frac{1}{2}|m| \rfloor$
 \end{prop}
 \begin{proof}
 	First of all, we claim that it is sufficient to prove the statement for Case 1 in \ref{singular}, i.e., when the hyperplane is $s_i-s_j=m$ with $m\in\ZZ_{\geq 0}$. The arguments for other cases are analogous, with either some Young diagram transposed or the indexes $i$ and $j$ switched. Note that in Case 2, when transposing the Young diagram (i.e., $m<0$),  one ends up with the absolute value $|m|$ in the statement.
 	
 	We denote the singular element in \ref{singular} Case 1 with $j$-th component $((r+m)^r)$ and other components $\emptyset$ by $\nu_r$.
 	
 	Consider a partition $\nu$ with $\nu^{(j)}=\emptyset$. For any box $b\in\nu$, there is no box (including addable boxes) $b'\sim b$. So applying $\tilde{e}_z$ for all possible values of $z\in\CC/\kappa^{-1}\ZZ$, one can remove all $n$ boxes in $\nu$ one by one. Hence $p=n$, corresponding to $r=0$ in the proposition.
 	
 	An element $\nu$ has depth less than $n$ if and only if one comes across one of the $\nu_r$'s while removing boxes, no matter in what order the $\tilde{e}_z$'s for different $z$'s are applied. By properties of crystals, if $\tilde{e}_{z_1}\cdots\tilde{e}_{z_k}\nu=\nu_r$ for some sequence of $z_1,\cdots,z_k$ and for some $r\geq 1$, then for any $\sigma\in\fS_k$, $\tilde{e}_{z_{\sigma(1)}}\cdots\tilde{e}_{z_{\sigma(k)}}\nu=\nu_r$ given that $\tilde{e}_{z_{\sigma(1)}}\cdots\tilde{e}_{z_{\sigma(k)}}\nu \neq 0$. So the order makes no difference. Since $|\nu_r|=(r+m)r$, $p(\nu)=n-(r+m)r$. Finally $\nu_r\subseteq \nu$ implies $(r+m)r\leq n$, which gives $r\leq \sqrt{n+\frac{1}{4}m^2}-\frac{1}{2}m$.
 \end{proof} 	
 
 Now our goal to find the depth of a given $\nu\in\cP_{\kappa,\ss}$ reduces to find the rectangle $\nu_r$ such that $\nu$ goes to $\nu_r$ after a sequence of crystal operators $\tilde{e}$. As explained in the proof of Proposition \ref{n-rec}, the crystal operators commute under some conditions. Therefore when $\ss$ lies on $s_i-s_j=m>0$, we can first remove boxes one-by-one in the $j$-th component until each removable box has an addable box in the $i$-th component equivalent to it with respect to $\sim$. Then we can remove one box from the $i$-th component and repeat the previous step. This procedure can continue until there is no removable box in the $i$-th component, i.e., the $i$-th component is empty, and there is a rectangle left in the $j$-th component, with the number of rows equal to $r$, which is exactly $\nu_r$. 
 
 To find this $r$, for each addable box $b$ to $\nu^{(i)}$, find the right-top $b'\in\nu^{(j)}$ (not necessarily removable) such that $b'\sim b$, and assign the number $r_b:=\max\{y(b')-y(b)+1,0\}$, where $y$ denotes the number of the row containing the corresponding box. If there is no such $b'$, assign $r_b=0$. This $r_b$ indicates how many rows are left in $\nu^{(j)}$ when there is no box in the $i$-th component equivalent to $b$ and each removable box in the $j$-th component has an equivalent addable box in the $i-$th component. Since we stop once we hit some $\nu_r$, this $r$ should be the maximum among those $r_b$. So far we have proved the following result.
 \begin{theorem}\label{conclusion}
 	 	Under the notation above, $$r=\max\{r_b:b\text{ runs over all addable boxes to }\nu^{(i)}\}$$ and $p(\nu)=n-r(r+|m|)$.
 \end{theorem}
 
 More generally, 
 Theorem \ref{conclusion} holds 
 for any parameter $\cc$ Weil generic on the hyperplane $s_i-s_j=m+\frac{t}{\kappa}$ with $i<j$, except if $t<0$. In the latter case, one needs to switch $i$ and $j$ for the whole argument.
 
\section{Two-sided Ideals of quantized quiver varieties}\label{ideal}
In this section, we consider the quantized quiver variety $\cA_\lambda^\theta(v,w)$ and its algebra $\cA_\lambda(v,w)$ of global sections, and the restriction functor $\bullet_{\dagger,x}$ constructed in \cite[Section 3.3]{bl}. In particular, we apply these techniques to obtain descriptions of two-sided ideals in $\cD_\lambda(\Gr(v,w))$, the algebra of twisted differential operators on the grassmannian variety, and in $eH_\cc e$, the spherical subalgebra of RCA. The two algebras can both be realized as $\cA_\lambda(v,w)$ for some $(Q,v,w)$ and $\lambda$.
 \subsection{Harish-Chandra bimodules and Restriction functors}\label{HC}
 In this section, we recall some preliminaries on Harish-Chandra (HC, for short) bimodules over quantized quiver varieties, and the construction of restriction functor $\bullet_{\dagger,x}$, following \cite[Section 3]{bl}.
 
 Analogously to the HC bimodules over the universal enveloping algebras of semisimple Lie algebras, one can define HC bimodules over arbitrary almost commutative filtered algebras, (see \cite{Lw,Lcompletion,gi,bpw} for example). Let $\cA=\bigcup_{i\geq 0}\cA^{\leq i}$, $\cA'=\bigcup_{i\geq 0}\cA'^{\leq i}$ be $\ZZ_{\geq 0}$-filtered algebras with the Lie bracket having degree $-1$, such that their associated graded $\gr\cA$, $\gr\cA'$ are identified with the same $A$, which is a finitely generated commutative graded Poisson algebra. We call an $\cA'$-$\cA$-bimodule $\cB$ {\it Harish-Chandra (HC)} if it can be equipped with a bimodule $\ZZ$-filtration bounded from below $\cB=\bigcup_i\cB^{\leq i}$, such that $\gr\cB$ is a finitely generated $A$-module (meaning, in particular, that the left and right actions of $A$ coincide). Such a filtration on $\cB$ is called {\it good}. We remark that every HC bimodule is finitely generated both as a left and as a right module, and that the support of $\gr\cB$ in $\Spec{A}$, which is called the {\it associated variety} of $\cB$ and denoted by $V(\cB)$, is independent of the choice of filtrations on $\cB$. Note that $V(\cB)$ is a Poisson subvariety of $\Spec{A}$.

 We will take $\cA=\cA'=\cA_\lambda(v,w)$  (or its universal version $\cA_\fP(v,w)$, to be defined below), where the filtration on $\cA$ is induced from the filtration on $D(R)$ by the order of differential operators. Then $A:=\CC[\cM(v,w)]$ (or $\CC[\cM_\fp(v,w)]$, to be defined below), so that $\gr\cA=\gr\cA'=A$. 

 In order to introduce the functor $\bullet_{\dagger,x}$, we need a decomposition of the formal neighborhood of the completion $\cA_\fP(v,w)^{\wedge_x}_\hbar$. We first consider the decomposition in the classical level.
 
 Set $\fp:=\CC^{Q_0}\simeq(\fg^*)^G$ to be the parameter space of Hamiltonian reductions and consider the families of quiver varieties $\cM_\fp^0(v,w):=\mu^{-1}(\fg^{*G})//G$, $\cM_\fp^\theta(v,w):=\mu^{-1}(\fg^{*G})^{\theta-ss}//G$ and $\cM_\fp(v,w):=\Spec(\CC[\cM_\fp^\theta(v,w)])$. Pick a point $x\in \cM_\fp^0(v,w)$. The following descriptions of the formal neighborhood $\cM_\fp^0(v,w)^{\wedge_x}$ and of the scheme $\cM_\fp^\theta (v,w)^{\wedge_x}:=\cM_\fp^0(v,w)^{\wedge _x}\times_{\cM_\fp^0(v,w)}\cM_\fp^\theta(v,w)$ are due to Nakajima, \cite[Section 6]{n1}.
 
 Let $r\in T^*R$ be a point with closed $G$-orbit mapping to $x$. Then $r$ is a semisimple representation of the double quiver $\overline{Q^w}$, where the quiver $Q^w$ is obtained by adjoining a vertex $\infty$ to $Q$ and connecting each vertex $i\in Q_0$ to $\infty$ with $w_i$ arrows. Then $\overline{Q^w}$ is obtained by adding an opposite arrow to each existing arrow of $Q^w$. The dimension of $r$ as a representation of $\overline{Q^w}$ is $(v,1)$. So $r$ can be decomposed into the sum $r=r_0\oplus r_1\otimes U_1\oplus\cdots \oplus r_k\otimes U_k$, where $r_0$ is the irreducible representation with dimension vector $(v^0,1)$, $r_1,\cdots,r_k$ are pair-wise non-isomorphic irreducible representations with dimensions $(v^i,0)$ where $i=1,\cdots,k$, and $U_i$ is the multiplicity space of $r_i$. In particular, the stabilizer $G_r$ of $r$ is $\Pi_{i=1}^kGL(U_i)$.
  
 We define a new framed quiver $(\hat{Q},\hat{v},\hat{w})$ following \cite[2.1.6]{bl}. The set of vertices $\hat{Q}_0$ is taken as $\{1,\cdots,k\}$ and $\hat{v}=(\dim U_i)_{i=1}^k$. The number of arrows between $i,j\in\{1,\cdots,k\}$ is $-(v^i,v^j)$ if $i\neq j$, and $p(v^i):=1-\frac{1}{2}(v^i,v^i)$ if $i=j$. The framing is $\hat{w}_i=w\cdot v^i-(v^0,v^i)$. We also need to add some loops at $\infty$ but those are just spaces with trivial $G_r$-actions. The orientation of $\hat{Q}$ is chosen such that the $G_r$-modules $R_x\oplus \fg/\fg_r$ and $R$ are isomorphic up to a trivial summand (although this choice may violate the condition that the vertex $\infty$ in $\hat{Q}$ is a sink). 
 
 The construction of $(\hat{Q},\hat{v},\hat{w})$ implies a decomposition $T^*R=T^*R_x\oplus T^*(\fg/\fg_r)\oplus R_0$, where $R_x=R(\hat{Q},\hat{v},\hat{w})$ and $R_0$ is a symplectic vector space with trivial $G_r$-action. The decomposition leads to an equality of formal Poisson schemes (\cite[2.1.6]{bl})
 $$ \cM_\fp^0(v,w)^{\wedge_x}=\hat{\cM}_\fp^0(\hat{v},\hat{w})^{\wedge_0}\times R_0^{\wedge_0}, $$
 where the superscript $\bullet^{\wedge_x}$ means the completion near $x$. Consider the restriction map $re: \fp=\fg^{*G}\to \hat{\fp}=\fg_r^{*G_r}$. We set $\hat{\cM}_\fp^0(\hat{v},\hat{w}):=\fp\times_{\hat{\fp}}\hat{\cM}_{\hat{\fp}}^0(\hat{v},\hat{w})$.
 
 We also have a similar decomposition for the smooth quiver varieties. 
 $$ \cM_\fp^\theta(v,w)^{\wedge_x}=(\hat{\cM}_\fp^\theta (\hat{v},\hat{w})\times R_0)^{\wedge_0}, $$
 where $\cM_\fp^\theta(v,w)^{\wedge_x}:=\cM_\fp^0(v,w) ^{\wedge_x}\times_{\cM_\fp^0(v,w)}\cM_\fp^\theta(v,w)$ and on the right hand side we slightly abuse the notation and write $\theta$ for the restriction of $\theta$ to $G_r$.
  
 On the quantized level, we consider the sheaf of $\CC[\fP]$-algebras, with $\fP:=\CC^{Q_0}$ denoting the affine parameter space for quantum Hamiltonian reductions,  $$\cA^\theta_{\fP}(v,w):=[D(R)/D(R)\Phi([\fg,\fg])|_{(T^*R)^{\theta-ss}}]^G $$ 
 on $\cM^\theta_{\fp}(v,w)$, and the $\CC[\fP]$-algebra $\cA_{\fP}(v,w)=\Gamma(\cA^\theta_{\fP}(v,w))$. 
 
 For the framed quiver $(\hat{Q},\hat{v},\hat{w})$ constructed above, we define $\hat{\cA}^0_{\fP}(\hat{v},\hat{w})$ as follows. Let $\hat{\fP}=(\hat{\fg}^*)^{\hat{G}}$ be the parameter space for the quantizations associated to $(\hat{Q},\hat{v},\hat{w})$. Let us define an affine map $\hat{re}:\fP\rightarrow \hat{\fP}$ whose differential is the restriction map $re:(\fg^*)^G\to (\hat{\fg}^*)^{\hat{G}}$. Namely, we have an element $\varrho(v)\in \ZZ^{Q_0}$, defined by 
 $$\varrho(v)_k:=-\frac{1}{2}(\sum_{a,h(a)=k}v_{t(a)}-\sum_{a,t(a)=k}v_{h(a)}-w_k), $$
 and $\hat{\varrho}(\hat{v})$ defined analogously. Set 
 $$\hat{re}(\lambda):=re(\lambda-\varrho(v))+\hat{\varrho}(\hat{v}).$$
 Further, set $\hat{\cA}^0_{\fP}(\hat{v},\hat{w}):=
 \CC[\fP]\otimes_{\CC[\hat{\fP}]}\hat{\cA}^0_{\hat{\fP}}(\hat{v},\hat{w})$ and define $\hat{\cA}^\theta_{\fP}(\hat{v},\hat{w})$ in a similar way. 
 
 Recall that starting from a filtered algebra $\cA$, we can form the Rees algebra $\cA_\hbar:=\bigoplus_i\cA^{\leq i}\hbar^i$ which is graded with $\deg\hbar=1$. Consider the Rees sheaves and algebras $\cA^{\theta}_{\fP}(v,w)_\hbar$, $\cA_{\fP}(v,w)_\hbar$, $\cA_{\fP}^0(v,w)_\hbar$ defined for the filtrations by orders of differential operators. We can complete those at $x$ getting the algebras $\cA_{\fP}(v,w)_\hbar^{\wedge_x}$, $\cA^0_{\fP}(v,w)_\hbar^{\wedge_x}$ with $\cA_{\fP}(v,w)_\hbar^{\wedge_x}/(\hbar)=\CC[\cM_{\fp}(v,w)^{\wedge_x}]$, $\CC[\cM^0_{\fp}(v,w)^{\wedge_x}]\twoheadrightarrow \cA^0_{\fP}(v,w)_\hbar^{\wedge_x}/(\hbar)$ and the sheaf of algebras $\cA^{\theta}_{\fP}(v,w)_\hbar^{\wedge_x}$ on $\cM^{\theta}_{\fp}(v,w)^{\wedge_x}$ obtained by the $\hbar$-adic completion of $\cA^0_{\fP}(v,w)_{\hbar}^{\wedge_x}\otimes_{\cA^0_{\fP}(v,w)_\hbar}\cA^\theta_{\fP}(v,w)_\hbar$. Note that $\cA^{\theta}_{\fP}(v,w)_\hbar^{\wedge_x}/(\hbar)=\cO_{\cM^{\theta}_{\fp}(v,w)^{\wedge_x}}$. We have the following result in \cite[Lemma 3.7]{bl} and \cite[Lemma 6.5.2]{Lquantization}.
 
 \begin{lemma}\label{decomp}
	We have the following decompositions.
 	\begin{align*}
 	&\cA^0_{\fP}(v,w)_\hbar^{\wedge_x}=	\hat{\cA}^0_{\fP}(v,w)_\hbar^{\wedge_0}\widehat{\otimes}_{\CC[[\hbar]]}\mathbf{A}_\hbar^{\wedge_0},\\
 	&\cA^\theta_{\fP}(v,w)_\hbar^{\wedge_x}=
 	\left(\hat{\cA}^\theta_{\fP}(v,w)_\hbar\otimes_{\CC[[\hbar]]}\mathbf{A}_\hbar\right)^{\wedge_0},\\
 	&\cA_{\fP}(v,w)_\hbar^{\wedge_x}=
 	\hat{\cA}_{\fP}(v,w)_\hbar^{\wedge_0}\widehat{\otimes}_{\CC[[\hbar]]}\mathbf{A}_\hbar^{\wedge_0}.
 	\end{align*}
  Here $\mathbf{A}$ denotes the Weyl algebra of the symplectic vector space $R_0$, defined as 
  $$\mathbf{A} = T(R_0)/(u\otimes v-v\otimes u - \omega(u,v)),$$ 
  and $\mathbf{A}_\hbar$ denotes the homogenized Weyl algebra.
 \end{lemma}
 
With the preparation above, one can construct the functor
$$\bullet_{\dagger,x}: \operatorname{HC}(\cA_{\fP}(v,w)) \rightarrow\operatorname{HC}(\hat{\cA}_{\fP}(\hat{v},\hat{w})),$$
for $x\in \cM_{\fp}(v,w)$, as in \cite[3.3.3]{bl}. We are not going to recall the technical details of the construction, but only the following properties, following \cite[3.4]{bl}.
\begin{prop}\label{dagger}
 	The following are true.
	\begin{itemize}
 		\item[a.] The functor $\bullet_{\dagger,x}$ is exact and $\CC[\fP]$-linear.
 		\item[b.] Let $\cB$ be a HC $\cA_\fP(v,w)$-bimodule. The associated variety of $\cB_{\dagger,x}$ is uniquely characterized by $V(\cB_{\dagger,x})\times\cL^{\wedge_x}=V(\cB)^{\wedge_x}$, where $\cL$ is the symplectic leaf through $x$. A similar claim holds for HC $\cA_\fP^0$-bimodules.
 		\item[c.] There is a right adjoint functor $\bullet^{\dagger,x}: \HC_{fg}(\hat{\cA}_{\fP}(\hat{v},\hat{w})) \to \HC(\cA_{\fP}(v,w))$ of $\bullet_{\dagger,x}$, where $\HC_{fg}(\hat{\cA}_{\fP}(\hat{v},\hat{w}))$ denotes the subcategory of $\HC(\hat{\cA}_{\fP}(\hat{v},\hat{w}))$ consisting of all bimodules $\mathcal{B}$ that are finitely generated over $\CC[\fP]$.
	\end{itemize}
 \end{prop} 
  
 \subsection{Chains of ideals: Grassmannian case}
 Now we apply the functor $\bullet_{\dagger,x}$ to the category $\HC(\cA_\lambda(v,w))$ constructed from the simplest quiver,  i.e., the quiver consisting of only one vertex and no loops. In this case, the algebra $\cA_{\lambda}(v,w)$ can be identified with the algebra $\cD_\lambda(\Gr(v,w))$ of $\lambda$-twisted differential operators on the Grassmannian variety. First we identify the framed quiver ($\hat{Q}$, $\hat{v}$, $\hat{w}$). Then we use the functor $\bullet_{\dagger,x}$ to obtain a description of two-sided ideals in $\cD_\lambda(\Gr(v,w))$, which are submodules of the regular module (which is HC) and hence HC. Assume, for convenience, that $w>2v$, $\theta>0$ (so that $(\lambda,\theta)$ is generic for any $\lambda$).
 
 We first look at the situation with $\lambda=0$. We will need the following classical result from Algebraic Geometry.
 \begin{lemma}\label{cohom}
 	When $v$ and $w$ are fixed, $H^i(\Gr(v,w),\cO(n))=0$ for any $i\in\ZZ$ if and only if $1-w\leq n\leq -1$, with $\cO$ denoting the trivial line bundle on $\Gr(v,w)$.
 \end{lemma}

 Now we are ready to prove the following result.
 \begin{prop}\label{regulargrass}
 	The algebra $\cD(\Gr(v,w))$ (with $w>2v$) has a chain of two-sided ideals 
 	$$ \{0\}= \cI_0\subsetneq \cI_1 \subsetneq \cdots \subsetneq \cI_v \subsetneq \cI_{v+1}=\cD(\Gr(v,w)). $$
 	Furthermore, $\cI_i$ with $i=0,\cdots,v+1$ exhaust all two-sided ideals of $\cD(\Gr(v,w))$.
 \end{prop}
 \begin{proof}
 	The following computation identifies the slice quiver $\hat{Q}$ for the quiver $Q$ consisting of only one vertex and no loops. 
 	
 	A semi-simple representation $r$ of the double quiver $\overline{Q^w}$, where $Q^w$ has two vertices $\{1,\infty\}$ and $w$ arrows from $1$ to $\infty$, with the dimension $(v,1)$ can be decomposed as $r=r_0\oplus r_1\otimes \CC^s$, where $r_0$ is an irreducible representation of dimension $(v-s,1)$ and $r_1$ is irreducible of dimension $(1,0)$ with the multiplicity $s$ satisfying $0\leq s\leq v$. Note that $p((v-s,1))=(w-v+s)(v-s)\geq 0$, so $(v-s,1)$ is indeed a root of $Q^w$ for any $s$ like above. So the slice quiver $\hat{Q}$ consists of 1 vertex and no loops because $p(\hat{v})=0$. The dimension is $\hat{v}=s$ and the framing is $\hat{w}=w\cdot 1-(v-s,1)=w-2v+2s$. Note that in this case, $re:(\fg^*)^G=\CC\to \CC$ is the identity map and thus $\hat{re}(0)=v-s\geq 0$. We obtain $(v+1)$ possible slice algebras denoted by $\cA_s:=\cA_{v-s}(s,w-2v+2s)=\cD_{v-s}(\Gr(s,w-2v+2s))$, with $s=0,\cdots,v$.
 
 	Note that $\cA_s=\cD_{v-s}(\Gr(s,w-2v+2s))$
 	has a unique non-zero finite-dimensional irreducible representation $L_s = H^0(\Gr(s,w-2v+2s),\cO(v-s))$. Then $\cI'_s:=\Ann(L_s)$ is a (maximal) two-sided ideal of finite codimension in $\cA_s$, where $\Ann(\cdot)$ denotes the left annihilator (which coincides with the right annihilator) of $L_s$ in $\cD_{v-s}(\Gr(s,w-2v+2s))$. Consider the right adjoint functor $\bullet^{\dagger,x}$, for $x$ being the image of $r$ in $\cM(v,w)=\Spec\CC[\cM_0^\theta(v,w)]$ with generic $\theta$. Let $\cI_s$ denote the kernel of the natural map
 	$$ \cD(\Gr(v,w))\to (\cA_s/\cI_s')^{\dagger,x}. $$ 
 	These are two-sided ideals in $\cD(\Gr(v,w))$ with the associated varieties $V(\cD(\Gr(v,w))/\cI_s)=\overline{\cL_s}$ (see the proof of \cite[Lemma 10.10]{bl}). Here $\cL_s$ denotes the symplectic leaf in $\cM(v,w)$, which can be described as follows. 
 	
 	There is a map $\mu^{-1}(0)\to \gln_w$ given by $(i,j)\mapsto ij$ with $i\in\Hom_\CC(V,W)$ and $j\in\Hom_\CC(W,V)$. It induces an embedding $\Spec\CC[\cM_0^\theta(v,w)] \hookrightarrow \gln_w$, whose image consists of all $w\times w$-matrices with square 0 and of rank not exceeding $v$. The symplectic leaf $\cL_s$ consists of those of rank $s$ and the closure $\overline{\cL_s}$ consists of those of rank not exceeding $s$, i.e., $\overline{\cL_s}=\bigcup_{r\leq s}\cL_r$.
 		
 	By construction, $\overline{\cL_s}\subsetneq \overline{\cL_{s+1}}$, therefore $(\cI_s)_{\dagger,x}$ is a proper ideal in $\cA_{s+1}$ and hence contained in a maximal ideal. Recall that $\cD(\Gr(v,w))$ is isomorphic to a quotient of the universal enveloping algebra $\mathcal{U}(\sln_w)$ (\cite{bb}). So $\cI'_{s+1}$ is the unique maximal ideal in $\cA_{s+1}$ because its preimage in the central reduction algebra $\mathcal{U}_{v-s-1}(\gln_{w-2v+2s+2})$ is the unique maximal ideal. So $\cI_s$ lies in the kernel of $\cD(\Gr(v,w))\to (\cA_{s+1}/\cI_{s+1}')^{\dagger,x}$, which is $\cI_{s+1}$, i.e., $\cI_s\subsetneq \cI_{s+1}$. So we have obtained a chain of ideals
 	$$  \cI_0\subsetneq \cI_1 \subsetneq \cdots \subsetneq \cI_v. $$
 	Further, note that $\cA_0=\cD(\Gr(0,w-v))=\CC$, so $\cI_0=\{0\}$. On the other hand, $\cA_v=\cD(\Gr(v,w))$ and the restriction functor $\bullet_{\dagger,x}:\cD(\Gr(v,w))\to \cA_v$ is just the identity so $\cI_v$ is the maximal ideal in $\cD(\Gr(v,w))$.
 	
 	Finally it remains to verify that any two-sided ideal $\cI\subseteq\cD(\Gr(v,w))$ coincides with some $\cI_s$. In fact, $V(\cD(\Gr(v,w))/\cI)$ is a Poisson subvariety, hence a union of symplectic leaves in $\cM_0(v,w)$. So $V(\cD_\lambda(\Gr(v,w))/\cI)=\overline{\cL_s}$ for some $0\leq s\leq v$. By construction, $\cI_s$ is the maximal ideal with associated variety $\overline{\cL_s}$. It follows that $\cI\subseteq\cI_s$. Since any finite dimensional representation of $\mathcal{U}(\sln_w)$ is completely reducible , the same is true for $\cD(\Gr(v,w))$. 
 	
 	Suppose $V(\cI_s/\cI)=\overline{\cL_r}$ for some $0\leq r\leq s$. Consider the restriction functor $$\bullet_{\dagger,x}:\HC(\cD(\Gr(v,w)))\to \HC(\cD_{v-r}(\Gr(r,w-2v+2r))).$$ It sends $\cI_s/\cI$ to a finite dimensional module $(\cI_s/\cI)_{\dagger,x}=(\cI_s)_{\dagger,x}/\cI_{\dagger,x}$. So we can assume that $\cI_s/\cI$ is finite dimensional.
 	
 	Recall that $\cD_\lambda(\Gr(v,w))$ is a quotient of $\mathcal{U}(\sln_w)$. So the two-sided ideals $\cI\subseteq\cI_s\subseteq \cD_\lambda(\Gr(v,w))$ give rise to ideals $I\subseteq I_s$ in $\mathcal{U}(\sln_w)$ containing the kernel of the map $\mathcal{U}(\sln_w)\twoheadrightarrow \cD(\Gr(v,w))$. 
 	Furthermore, the two-sided ideals in $\mathcal{U}(\sln_w)$ are in bijection with submodules of the Verma module	$\Delta(0)$ via $I\mapsto I\Delta(0)$ (\cite{bb}). So we have $I\Delta(0)\subseteq I_s\Delta(0)\subseteq \Delta(0)$. Therefore the quotient $I_s\Delta(0)/I\Delta(0)$, if nontrivial, should be filtered by some composition factors $L(\mu)$ (the unique simple quotient of the corresponding Verma module) of $\Delta(0)$, with $\mu<0$ being a weight of $\sln_w$ lying in the same dot action orbit by the Weyl group $\fS_w$. In particular, $\mu<0$ implies that $\mu$ is not dominant, and that $L(\mu)$ is infinite-dimensional, which leads to a contradiction. So $I=I_s$ and $\cI=\cI_s$.
 \end{proof}

 For any regular integral $\lambda$, $\cD_\lambda(\Gr(v,w))$ and $\cD(\Gr(v,w))$ are Morita equivalent. Now we proceed to the case when $\lambda$ is singular. From Lemma \ref{cohom} one can deduce that the singular locus for the parameter $\lambda$ is $\{1-w,2-w,\cdots,-1\}$. The following is the main result of this section.
 
 \begin{theorem}\label{singulargrass}
 	The algebra $\cD_\lambda(\Gr(v,w))$ ($w>2v$) with singular $\lambda\in\{1-w,2-w,\cdots,-1\}$ has a chain of two-sided ideals 
 	$$ \{0\}= \tilde{\cI}_0\subsetneq \tilde{\cI}_1 \subsetneq \cdots \subsetneq \tilde{\cI}_p \subsetneq \tilde{\cI}_{p+1}=\cD_\lambda(\Gr(v,w)), $$
 	where $p=\max\{\lambda+v,v-w-\lambda,0\}$.
 	Furthermore, $\tilde{\cI}_i$ with $i=0,\cdots,p+1$ exhaust all two-sided ideals of $\cD_\lambda(\Gr(v,w))$.
 \end{theorem}
 \begin{proof}
 	The statement can be proved analogously to Proposition \ref{regulargrass}, except that some slice algebras have no nonzero finite dimensional representations, and therefore cannot induce a two-sided ideal of $\cD_\lambda(\Gr(v,w))$.
 	
 	The quivers for the slice algebras are still those corresponding to $\Gr(s,w-2v+2s)$ with $s=0,1,\cdots,v$. The parameter for $\Gr(s,w-2v+2s)$ is now $\lambda+v-s$. Note that the global section functor $\Gamma$ induces a quotient functor sending twisted D-modules on the grassmannian to $\mathcal{U}_{\lambda'}(\sln_w)$-modules (similarly to the Beilinson-Bernstein theorem for Lie algebras). For a simple $\cD_\lambda(\Gr(v,w))$-module $L$, $\Gamma(L)$ is a nonzero finite dimensional representation if and only if $\Supp(L)\subseteq \Gr(v,w) \subseteq T^*\Gr(v,w)$. So $L$ is a line bundle. Consider the twist $\lambda+v-s$. There exists a proper two-sided ideal $\tilde{\cI}_s$ of finite codimension if and only if $$ H^*(\Gr(s,w-2v+2s),\cO(\lambda+v-s))=\bigoplus_{i\in\ZZ}H^i(\Gr(s,w-2v+2s),\cO(\lambda+v-s))\neq 0. $$  
 	
 	By Lemma \ref{cohom}, $H^*(\Gr(s,w-2v+2s),\cO(\lambda+v-s))\neq 0$ if and only if $\lambda+v-s\geq 0$ or $\lambda+v-s\leq -(w-2v+2s)$, i.e., $s\leq \max\{\lambda+v,v-w-\lambda\}$. 
 	
 	So we have a chain of ideals
 	$$ \{0\}= \tilde{\cI}_0\subsetneq \tilde{\cI}_1 \subsetneq \cdots \subsetneq \tilde{\cI}_{\max\{\lambda+v,v-w-\lambda\}} \subsetneq \tilde{\cI}_{\max\{\lambda+v,v-w-\lambda\}+1}=\cD_\lambda(\Gr(v,w)), $$
 	when $\max\{\lambda+v,v-w-\lambda\}\geq 0$. In the case when $\max\{\lambda+v,v-w-\lambda\}<0$, there are no finite dimensional representations, hence no ideals of finite codimension. So the only proper ideal is $\{0\}$ and we have $\{0\}= \tilde{\cI}_0\subsetneq \tilde{\cI}_1=\cD_\lambda(\Gr(v,w))$. Therefore we take $p=\max\{\lambda+v,v-w-\lambda,0\}$ and the theorem is true.
 \end{proof}

 \subsection{Chains of ideals: Cherednik case}
 In this section, we consider the cyclic quiver $Q$ of the affine type $\tilde{A}_{\ell-1}$ with dimension $v=n\delta$ and framing $w=\epsilon_0$ and study the restriction functor $\bullet_{\dagger,x}: \HC(\cA_\fP(v,w)) \to \HC(\hat{\cA}_{\fP}(\hat{v},\hat{w}))$. Recall by Theorem \ref{quancm}, $\cA_{\lambda^q}(v,w)\cong eH_\cc e$, where the parameter $\lambda^q$ can be recovered from $\cc$ via the formula there. In particular, the singular $\lambda^q$ corresponds to the aspherical $\cc$.
 
 First we identify the quiver $\hat{Q}$ and the slice algebra $\hat{\cA}_{\fP}(\hat{v},\hat{w})$. Consider the direct sum decomposition of a semi-simple representation $r\in T^*R(Q,v,w)\cap \mu^{-1}(\lambda)$, where $\mu:T^*R(Q,v,w)\to \fg$ denotes the moment map. Here $\lambda$ parametrizes the classical quiver variety $\cM_\lambda(v,w)$ and is generic on the linear hyperplane
 $$ d_i-d_j=\ell mc_0 $$
 parallel to the aspherical hyperplane containing $\cc$, i.e., $d_i-d_j-\ell mc_0=k$ ($s_i-s_j=m+\frac{t}{\kappa}$ in terms of $(\kappa,\ss)$-parameter with $t=\frac{i-j-k}{\ell}$, as in Section \ref{singular}), where we assume $i<j$.
 So by Theorem \ref{classcalcm}, $\lambda=\lambda^c$ and $\cc$ lies on the hyperplane $ d_i-d_j=\ell mc_0$. 
 Therefore
 $$ \lambda^c=\frac{1}{\ell}(c_0\ell-d_0+d_{\ell-1},d_0-d_1,\cdots,d_{i-2}-d_{i-1},d_{i-1}-d_j-\ell mc_0,d_j-d_{i+1}+\ell mc_0,d_{i+1}-d_{i+2},\cdots,d_{\ell-2}-d_{\ell-1}). $$
 
 By definition, representations in $T^*R(Q^w,v,w)\cap \mu^{-1}(\lambda^c)$ are in one-to-one correspondence with the representations of the {\it deformed preprojective algebra} $\Pi^{\lambda^c}(Q^w)$, defined as
 $$ \Pi^{\lambda^c}(Q^w)=\CC\overline{Q^w}/(\sum_{a\in Q^w_1}[a,a^*]-\sum_{i\in Q_0}\lambda^c_i\epsilon_i-\lambda^c_\infty\epsilon_\infty), $$ 
 where $\CC\overline{Q^w}$ denotes the path algebra of the double quiver of $Q^w$ and $a^*$ denotes the opposite arrow of $a\in Q^w_1$. In \cite{cb}, Crawley-Boevey gave a description of the dimension vectors of simple representations of $\Pi^{\lambda^c}(Q^w)$: there exists a simple representation of dimension $\tilde{v}$ if and only if
 \begin{itemize}
 	\item $\tilde{v}$ is a positive root such that $(\lambda^c,\lambda^c_\infty)\cdot \tilde{v}=0$;
 	\item Proper Decomposition Condition: $p(\tilde{v})>\sum_{i=1}^kp(\tilde{v}_{(i)})$ for any proper decomposition $\tilde{v}=\sum_{i=1}^k\tilde{v}_{(i)}$ where $p(\tilde{v})=1-\frac{1}{2}(\tilde{v},\tilde{v})$ with $(\cdot,\cdot)$ denoting the symmetrized Tits form, and $\tilde{v}_{(i)}$ non-zero dimension vectors such that $(\lambda^c,\lambda^c_\infty)\cdot \tilde{v}_{(i)}=0$.
 \end{itemize}
 Here $\lambda^c_\infty$ denotes the parameter corresponding to the vertex $\infty$. Note that since $(\lambda^c,\lambda^c_\infty)\cdot (n\delta,1)=0$, we get $\lambda^c_\infty=-n\sum_{i=0}^{\ell-1}\lambda_i=-nc_0$.
 
 Clearly, a dimension vector of the form $(v^i,0)$ satisfies the above two conditions if and only if $v^i$ satisfies the similar conditions for $Q$. Also, the imaginary roots $n'\delta$ with $n'\in\ZZ_{>0}$ and $n'\leq n$ for $Q$ do not satisfy the proper decomposition condition. All positive real roots of type $\tilde{A}_{\ell-1}$ are of the form 
 $$ v'=(a,\cdots,a,a\pm 1,\cdots,a\pm 1,a,\cdots,a), $$
 with $a\in\ZZ_{> 0}$. So the only $v'$ such that $\lambda^c\cdot v'=0$ is 
 $$ v'=(|m|,\cdots,|m|,|m|+1,\cdots,|m|+1,|m|,\cdots,|m|) $$
 with $v'_{i+1}=\cdots=v'_j=|m|+1$ and other coordinates equal to $|m|$. It is easy to verify that $v'$ satisfies the proper decomposition condition. By \cite{cb}, for the real root $v'$ there exists a unique simple $\Pi^{\lambda^c}(Q^w)$-representation $r_1$ of dimension $v'$. Therefore the direct sum decomposition we are looking for should be $r=r_0\oplus r_1\otimes\CC^s$ corresponding to the decomposition $(n\delta,1)=(v'',1)+s(v',0)$ on the dimension level. By the result of \cite{cb}, 
 $$ (v'',1)=(n-s|m|,\cdots,n-s|m|,n-s|m|-s,\cdots,n-s|m|-s, n-s|m|,\cdots,n-s|m|,1) $$
 is a root if and only if 
 $$p((v'',1))=n-s|m|-s^2\geq 0,$$
 i.e., $0\leq s\leq\sqrt{n+\frac{1}{4}m^2}-\frac{1}{2}|m|$. Set $q=\lfloor \sqrt{n+\frac{1}{4}m^2}-\frac{1}{2}|m|\rfloor$.
 
 Therefore we get the quiver $\hat{Q}$, containing one vertex (corresponding to $r_1$ and dimension $(v',0)$) and no loops (due to $p((v',0))=0$). The quiver $\hat{Q}$ comes with the dimension $\hat{v}^s=s$ and the framing
 $$ \hat{w}^s=\epsilon_0\cdot v'-(v'',v')=|m|+2s. $$
 
 It remains to compute the new parameter $\hat{\lambda}^s$. The group $G_r$ embeds into $G=\GL_n^{Q_0}$ as the stabilizer of $r=r_0\oplus r_1\otimes\CC^s$. This embedding induces $\fg_r\hookrightarrow \fg$, which gives rise to 
 $$re:(\fg^*)^G\cong \CC^\ell\to(\fg_r^*)^{G_r}\cong \CC,$$
 sending $(x_0,\cdots,x_{\ell-1})$ to $\bar{x}$ defined by $\sum_{i=0}^{\ell-1}x_i\Tr X_i=\bar{x}\Tr X$, where $(X_0,\cdots,X_{\ell-1})\in \fg$ is the image of $X\in \fg_r$ under the embedding above. So 
 $$ \bar{x}=|m|(x_0+x_1+\cdots+x_i)+(|m|+1)(x_{i+1}+\cdots+x_j)+|m|(x_{j+1}+\cdots+x_{\ell-1})=(x_0,\cdots,x_{\ell-1})\cdot v'. $$
 Recall that the new parameter $\hat{\lambda}^s$ is obtained from $\lambda^q$ via the map $\hat{re}$ defined in Section \ref{HC}, where $\lambda^q$ is recovered from $\cc$ lying on $d_i-d_j-\ell mc_0=k$ $(i<j)$ via the formula in Theorem \ref{quancm}. Therefore $$ \hat{\lambda}^s=re(\lambda^q-\varrho(n\delta,\epsilon_0))+\hat{\varrho}(s,|m|+2s)=t-s, $$
 where $t:=\frac{i-j-k}{\ell}\in\ZZ$.
 
 So far we have obtained the slice algebra $\hat{\cA}_s:=\hat{\cA}_{\hat{\lambda}^s}(\hat{v}^s,\hat{w}^s)$ with $s=0,1,\cdots,q$, which can be identified with the algebra $\cD_{t-s}(\Gr(s,|m|+2s))$ of twisted differential operators on the Grassmannian variety.
 
 Now we consider the restriction functor $\bullet_{\dagger,x}:\HC(\cA_{\fP}(v,w))\to \HC(\hat{\cA}_\fP(\hat{v},\hat{w}))$. Take $\fP:=P$ as the hyperplane corresponding to (under the $\lambda$-$\cc$ correspondence as in Theorem \ref{classcalcm}) the aspherical hyperplane. Note that $\hat{re}(P)=t-q$. So the functor becomes $\bullet_{\dagger,x}:\HC(\cA_{P}(v,w))\to \HC(\CC[P]\otimes\cD_{t-q}(\Gr(q,|m|+2q)))$. We have the following result.
 
 \begin{theorem}\label{idealbijection}
 	In terms of the notation above, the restriction functor $\bullet_{\dagger,x}:\HC(\cA_P(v,w))\to \HC(\CC[P]\otimes\cD_{t-q}(\Gr(q,|m|+2q)))$ induces a bijection between the set of all two-sided ideals in $eH_\cc e$, where $\cc$ is Weil generic on the aspherical hyperplane $d_i-d_j-\ell mc_0=k$ as in Theorem \ref{dg} (up to $\pm$ signs on $m$ and $k$ to assure $i<j$), and that of all two-sided ideals in $\cD_{t-q}(\Gr(q,|m|+2q))$, where $q=\lfloor \sqrt{n+\frac{1}{4}m^2}-\frac{1}{2}|m|\rfloor$ and $t=\frac{i-j-k}{\ell}$.
 \end{theorem}	
\begin{proof}
The algebra $\cD_{t-q}(\Gr(q,|m|+2q))$ is defined over $\QQ$, while the algebra $\cA_{P}(v,w)$ is defined over a finite extension of $\QQ$. Denote the finite extension by $\mathbb{K}$. Let $\cD_{t-q}^\mathbb{K}(\Gr(q,|m|+2q))$, $\cA_{P}^\mathbb{K}(v,w)$ denote the corresponding $\mathbb{K}$-forms. Note that $\cA_{P}^\mathbb{K}(v,w)$ is a $\mathbb{K}[P]$-algebra. We have a $\mathbb{K}$-algebra homomorphism $\mathbb{K}[P]\to \CC$ of evaluation at $\lambda$. Since $\cc$ is Weil generic, so is $\lambda$. This homomorphism factors through the fraction field $\mathbb{K}(P)$. Set $$\cA_{\mathbb{K}(P)}(v,w):=\mathbb{K}(P)\otimes_{\mathbb{K}[P]}\cA_{P}^\mathbb{K[P]}(v,w),$$ $$\cD_{t-q}^{\mathbb{K}(P)}(\Gr(q,|m|+2q)):=\mathbb{K}(P)\otimes_{\mathbb{K}}\cD_{t-q}^\mathbb{K}(\Gr(q,|m|+2q)).$$ Note that 
$$\CC\otimes_{\mathbb{K}(P)}\cA_{\mathbb{K}(P)}(v,w) = \cA_\lambda(v,w), $$ $$\CC\otimes_{\mathbb{K}(P)}\cD_{t-q}^{\mathbb{K}(P)}(\Gr(q,|m|+2q)) = \cD_{t-q}(\Gr(q,|m|+2q)).$$

We equip $\cA_{P}^\mathbb{K}(v,w)$ with a filtration such that $\mathbb{K}[P]$ is in degree 0 (see the proof of \cite[Lemma 5.7]{bl}). So we can talk about HC $\cA_{\mathbb{K}(P)}(v,w)$-bimodules. Note that for such a bimodule $\cB$, $\CC\otimes_{\mathbb{K}(P)}\cB$ is a HC $\cA_\lambda(v,w)$-bimodule. We get a functor $$\bullet_{\dagger,x}: \HC(\cA_{P}^{\mathbb{K}}(v,w))\to \HC(\mathbb{K}[P]\otimes\cD_{t-q}^{\mathbb{K}}(\Gr(q,|m|+2q))).$$ It induces a functor $$\bullet_{\dagger,x}: \HC(\cA_{\mathbb{K}(P)}(v,w))\to \HC(\cD_{t-q}^{\mathbb{K}(P)}(\Gr(q,|m|+2q)))$$ by base change to $\mathbb{K}(P)$. We claim that the latter functor is faithful. Indeed, assume the contrary. Let $\cB\in\HC(\cA_{\mathbb{K}(P)}(v,w))$ be such that $\cB_{\dagger,x}=0$. We can choose a $\mathbb{K}[P]$-lattice $\cB_0\subseteq\cB$ that is a HC $\cA_{\mathbb{K}(P)}(v,w)$-bimodule. We have that $(\cB_0)_{\dagger,x}$ is $\mathbb{K}[P]$-torsion. By modifying $\cB_0$ we can assume $(\cB_0)_{\dagger,x}=0$. It means that $x$ does not lie in the associated variety of $\cB_0$. By the choice of $x$, $\cB_0$ is torsion over $\mathbb{K}[P]$. This is a contradiction with the choice of $\cB_0$.

In particular, the set of two-sided ideals in $\cA_{\mathbb{K}(P)}(v,w)$ embeds into the set of two-sided ideals in $\cD_{t-q}^{\mathbb{K}(P)}(\Gr(q,|m|+2q))$, preserving inclusions. Hence, the two-sided ideals in $\cA_{\mathbb{K}(P)}(v,w)$ form a chain. 

Let us show that $\bullet_{\dagger,x}$ induces a bijection between the sets of ideals. For this we need to show that the numbers of ideals coincide. Recall that the ideals in $\cD_{t-q}(\Gr(q,|m|+2q))$ are constructed as the kernels of the maps $\cD_{t-q}(\Gr(q,|m|+2q))\to (\cD_{t-s}(\Gr(s,|m|+2s))/\Ann L_s)^{\dagger,x}$, where $L_s$ is a finite dimensional representation. The same is true over $\mathbb{K}(P)$. But $\cD_{t-s}^{\mathbb{K}(P)}(\Gr(s,|m|+2s))$ is a slice algebra for $\cA_{\mathbb{K}(P)}(v,w)$ as well. So for each $s$ such that $\cD_{t-s}(\Gr(s,|m|+2s))$ has a finite dimensional representation, we have a two-sided ideal in $\cA_{\mathbb{K}(P)}(v,w)$. All the ideals are pair-wise distinct. So we indeed get a bijection between the sets of the two-sided ideals.

Now we need to show that there is a bijection between the sets of two-sided ideals in $\cA_\lambda(v,w)$ and $\cA_{\mathbb{K}(P)}(v,w)$. Recall that $\cA_\lambda(v,w)=\CC\otimes\cA_{\mathbb{K}(P)}(v,w)$. The algebra $\cA_\lambda(v,w)$ has only finitely many prime ideals by the appendix to \cite{es}. Therefore all of them are defined over $\mathbb{K}(P)$ for a suitable choice of $\mathbb{K}$. But all ideals in $\cA_{\mathbb{K}(P)}(v,w)$ coincide with their squares. It follows that all ideals in $\cA_\lambda(v,w)$ coincide with their squares, so are intersections of prime ideals and, therefore are defined over $\mathbb{K}(P)$. This finishes the proof. 
\end{proof}

 We remark that by Theorem \ref{dg}, one can show that $1-q\leq t\leq q-1$. Therefore $1-2q\leq t-q\leq -1$ and $t-q$ is a singular twist for $\Gr(q,|m|+2q)$. The following corollary can be obtained directly from Theorem \ref{idealbijection} and Theorem \ref{singulargrass}.
 \begin{cor}\label{cor}
 	For a Weil generic aspherical $\cc$ lying on the hyperplane $d_i-d_j-\ell mc_0=k$ as in Theorem \ref{dg}, all two-sided ideals in $eH_\cc e$ form a chain
 	$$ \{0\}=\tilde{\cJ}_0\subsetneq\tilde{\cJ}_1\subsetneq\cdots\subsetneq\tilde{\cJ}_p\subsetneq \tilde{\cJ}_{p+1}=eH_\cc e,  $$
 	where $p=\max\{t,0\}$ with 
 	$t=\frac{i-j-k}{\ell}$.
 \end{cor}
 
 \section{Description of $K_0(\cO^{sph}_\cc(W))$}\label{sphO}
 The category $\cO^{sph}_\cc$ was defined in \cite[3.1.3]{gl}.
 \begin{definition}
 	We say an $eH_\cc e$-module $M$ lies in the category $\cO^{sph}_\cc(W)$ if
 	\begin{itemize}
 		\item $M$ is finitely generated;
 		\item The Euler element $\hh^{sph}:=e\hh$ acts on $M$ locally finitely, where $\hh$ is defined in Section \ref{O};
 		\item the set of $\hh^{sph}$-weights on $M$ is bounded below, meaning that there exist $\alpha_1,\cdots,\alpha_k\in\CC$ such that the weight subspace $M_\alpha$ is zero whenever $\alpha-\alpha_i\notin \ZZ_{\geq 0}$ for all $i$.		
 	\end{itemize}
 \end{definition}
Clearly the assignment $M\mapsto eM$ defines a functor $\cO_\cc(W) \to \cO^{sph}_\cc(W)$ inducing a surjective map $K_0(\cO_\cc(W))\twoheadrightarrow K_0(\cO^{sph}_\cc(W))$. In this section, we describe the kernel of this map, i.e., all simples annihilated by $e=\frac{1}{|W|}\sum_{w\in W}w$, still under the assumption that $\cc$ is Weil generic on some aspherical hyperplane.

In Section \ref{supp}, we deduced that all possible supports of the simples $L_\nu\in\cO_\cc(n)$ with $\nu\in\cP_\ell(n)$ are of the form $\overline{X(W_{p,0})}$ with $p=n-r(|m|+r)$ ($0\leq r\leq q$). Denote $p_r:=n-r(|m|+r)$ and $\cJ_r:=\bigcap_\nu\Ann(L_\nu)$, where the intersection is taken over all $\nu\in\cP_\ell(n)$ such that $\dim\Supp(L_\nu)\leq p_r$. It is clear that $\cJ_r\subseteq H_\cc$ are two-sided ideals satisfying $\cJ_r\subsetneq \cJ_{r+1}$, and $\cJ_0$ is the intersection of annihilators of all simples, hence zero while $\cJ_q$ is the maximal ideal in $H_\cc$. Denote $\cJ_{q+1}:=H_\cc$ and we get a chain
\begin{align}\label{chainRCA}
\{0\}= \cJ_0\subsetneq \cJ_1 \subsetneq \cdots \subsetneq \cJ_q \subsetneq \cJ_{q+1}=H_\cc.
\end{align}

It follows that the idempotent $e\in \cJ_r$ and $e\notin\cJ_{r-1}$ for some $1\leq r\leq q+1$. We now explain how to find this $r$. For each $i=0,\cdots, q+1$, $e\cJ_i$ has to be an ideal of $eH_\cc e$ and hence coincides with some $\tilde{\cJ}_{i'}$ in Corollary \ref{cor}. Then we can prove the following.
\begin{prop}
	For any two-sided ideal $\cJ_i\subseteq H_\cc$, $e\in \cJ_i$ if and only if $e\cJ_i=eH_\cc e$.
\end{prop}
\begin{proof}
	Since $\cJ_i$ is a two-sided ideal, given $e\in \cJ_i$, for any $x\in H_\cc$, $xe\in \cJ_i$. So $exe\in e\cJ_i$. Thus we have $eH_\cc e=e\cJ_i$.
	
	Suppose $e \notin \cJ_i$. Then $e\notin e\cJ_i$ either. But evidently $e\in eH_\cc e$. So $eH_\cc e\neq e\cJ_i$.
\end{proof}
Combining Corollary \ref{cor}, (\ref{chainRCA}), and the proposition above, we have the following result.
\begin{cor}
	$e\notin \cJ_i$ and $e\cJ_i=\tilde{\cJ}_i$ for $i=0,\cdots,p=\max\{0,\frac{i-j-k}{\ell}\}$; $e \in \cJ_j$ and $e\cJ_j=eH_\cc e$ for $j=p+1,\cdots,q+1$. 
\end{cor}
\begin{proof}
	Recall that each two-sided ideal $\tilde{\cJ}_s\subseteq eH_\cc e=\cA_\lambda(v,w)$ is obtained as $\CC\otimes J_s$ where  $J_s$ is the kernel of the map $\cA_{\mathbb{K}(P)}(v,w)\to(\cD_{t-s}(\Gr(s,|m|+2s))/\Ann L_s)^{\dagger,x}$, where $L_s$ is the unique nonzero finite dimensional representation of the slice algebra $\cD_{t-s}(\Gr(s,|m|+2s))$. So $\dim V(eH_\cc e/\tilde{\cJ}_s) =\dim V(\cA_{\mathbb{K}(P)}(v,w)/J_s)$.
	It follows from \cite[Lemma 4.4, Remark 4.5]{Lholonomic} that
	$$\dim V(\cA_{\mathbb{K}(P)}(v,w)/J_s) =\dim\cL_s =2p_s,$$
	where $\cL_s$ denotes the symplectic leaf through $x$. By construction, $\dim V(H_\cc /\cJ_s)=2p_s$. So we have that $\dim V(eH_\cc e/\tilde{\cJ}_s) = \dim V(H_\cc /\cJ_s)$.
	
	Also note that if $e(H_\cc/\tilde{\cJ}_s)\neq 0$, then $$ \dim V(H_\cc/\cJ_s) =\dim V(e(H_\cc/\cJ_s)) = \dim V(eH_\cc e/(e\cJ_s))$$.
	
	Therefore we have $e\cJ_i=\tilde{\cJ}_i$ for $i=0,\cdots,p=\max\{0,\frac{i-j-k}{\ell}\}$. Other claims in the statement are straightforward.
\end{proof}
Further, considering the construction of the ideals in (\ref{chainRCA}), we obtain the following.
\begin{cor}
	Suppose the parameter $\cc$ for the cyclotomic RCA $H_\cc$ is Weil generic on the aspherical hyperplane $d_i-d_j-\ell mc_0=k$ (or in terms of the $\ss$-parameters $s_i-s_j=m+\frac{t}{\kappa}$, with $t=\frac{i-j-k}{\ell}$).
	Then the averaging idempotent $e\in H_\cc$ annihilates all simples in $\cO_\cc$ with supports of dimension not exceeding $n-(p+1)(|m|+p+1)$, where $p=\max\{0,t\}$. The kernel of the map $K_0(\cO_\cc(W))\twoheadrightarrow K_0(\cO^{sph}_\cc(W))$ is generated by all classes of such simples.
\end{cor}
So as the constant term $k$ (or $t$ in terms of $\ss$-parameters) in the equation of the aspherical hyperplane increases, the number of ideals in $eH_\cc e$ decreases, so does the number of simples in $\cO_\cc(W)$ annihilated by $e$. When $k$ reaches the boundary such that $p=0$, there are no proper nontrival ideals at all. 

Finally we remark that our results prove the Grassmannian case and Cherednik case when $\cc$ is Weil generic and aspherical in \cite[Conjecture 11.8]{bl}.

\end{document}